\documentclass[11pt]{amsart}
\usepackage{geometry}               
\usepackage{graphicx}
\usepackage{tikz}
\usepackage{amssymb}
\usepackage{bbm}
\usepackage{pdfsync}
\usepackage{hyperref}
\usepackage{verbatim} 
\usepackage{epstopdf}
\usepackage{color}
\def\R{\mathbb{R}}

\def\N{\mathbb{N}}

\def\Co{\mathbb{C}}

\renewcommand{\d}{\,{\rm d}} 
\newcommand{\sph}[1]{\mathbb{S}^{#1}}

\providecommand{\les}{\lesssim}

\newtheorem{theorem}{Theorem}[section]
\newtheorem{corollary}[theorem]{Corollary}

\newtheorem{remark}[theorem]{Remark}
\newtheorem{proposition}[theorem]{Proposition}
\newtheorem{lemma}[theorem]{Lemma}

\numberwithin{equation}{section}

\title{The Stein--Tomas inequality under the effect of symmetries}

\author{Rainer Mandel}
\address{
        Karlsruhe Institute of Technology\\
        Institute for Analysis\\
        Englerstrasse 2\\
        D-76131 Kalrsruhe, Germany}
\email{rainer.mandel@kit.edu}

\author{Diogo Oliveira e Silva}
\address{ 
Departamento de Matemática\\ 
Instituto Superior Técnico\\
Av. Rovisco Pais\\ 
1049-001 Lisboa, Portugal} 
\email{diogo.oliveira.e.silva@tecnico.ulisboa.pt}

\begin{document}

\subjclass[2010]{42B10}
\keywords{Stein--Tomas inequality, sharp restriction theory, symmetry, orthogonal group, maximizers, maximizing sequences.}
\begin{abstract}
We prove new Fourier restriction estimates to the unit sphere $\sph{d-1}$ on the class of $O(d-k)\times O(k)$-symmetric functions, for every $d\geq 4$ and $2\leq k\leq d-2$.
As an application, we establish the existence of maximizers for the endpoint Stein--Tomas inequality within that class.
Moreover, we construct examples showing that the range of Lebesgue exponents in our estimates is sharp.
\end{abstract}

\maketitle

\section{Introduction}

The Fourier restriction conjecture predicts the validity of the estimate
\begin{align}\label{eq:FRC}
\left(\int_{\mathbb{S}^{d-1}} |\widehat{f}(\omega)|^q\, {\rm d}\sigma(\omega)\right)^{\frac1q}\leq C(d,p,q) \|f\|_{L^p(\R^d)},
\end{align}
as long as the dimension $d\geq 2$ of the ambient Euclidean space $\R^d$ and the Lebesgue exponents $p,q\in[1,\infty]$ satisfy the following conditions:
\begin{equation}\label{eq:FRCNumerology}
 \frac1p>\frac{d+1}{2d} \quad\text{and}\quad \frac{d+1}{p}+\frac{d-1}{q}\geq d+1.
 \end{equation}
Integration on the left-hand side of~\eqref{eq:FRC} is with respect to the usual surface measure $\sigma$ on
the unit sphere $\sph{d-1}:=\{\omega\in\R^d: |\omega|=1\}$.
The restriction conjecture has a rich history, and is remarkable in its numerous connections and applications. 
It exhibits deep links to Bochner--Riesz summation methods
 and to decoupling phenomena for the Fourier transform, and is known to imply the Kakeya conjecture.
Despite the great deal of attention received by this circle of problems during the past four decades, the restriction conjecture has been established only when $d=2$ (see \cite{Fef} for the non-endpoint case $p'>3q$, and \cite{CS72, Zy74} for the endpoint $p'=3q$) and remains an open question in dimensions $d\geq 3$. 
For further information on the restriction problem, we refer the interested reader to the survey \cite{Ta04}
and the recent account \cite{St19}.

\medskip

The special case $q=2$ in~\eqref{eq:FRC} is well understood and of particular importance.
If $d\geq 2$ and $1\leq p\leq \frac{2(d+1)}{d+3}$, then the classical Stein--Tomas inequality \cite{St93,
To75} states the existence of a constant $C(d,p)<\infty$, such that
\begin{equation}\label{eq:TSgen}
\left(\int_{\sph{d-1}} |\widehat{f}(\omega)|^2\d\sigma(\omega)\right)^{\frac12}\leq C(d,p) \|f\|_{L^p(\R^d)},
\end{equation}
for every $f\in L^p(\R^d)$.
A well-known construction of Knapp (see \cite{St77}) reveals that the range of Lebesgue exponents is sharp in this case.
In particular, estimate \eqref{eq:TSgen} is of endpoint type when $p=\frac{2(d+1)}{d+3}$, in the sense that it becomes false if either of the exponents $2,\frac{2(d+1)}{d+3}$ is increased.
The Stein--Tomas inequality finds applications in harmonic analysis and PDEs.
In particular, it underlies most of the early progress towards the Fourier restriction conjecture; see \cite{Ta04}.
The argument is robust enough to be applied to other manifolds, e.g.\@ to the paraboloid, the
cone, and the hyperboloid, in which case it implies the foundational Strichartz
estimates for the Schr\"odinger, Wave, and Klein--Gordon equations, respectively;  see \cite{St77}.
Moreover, inequality  \eqref{eq:TSgen}  has been generalized to a variety of other contexts, 
and found surprising applications ranging from fractal geometry \cite{Mo00} to number theory \cite{Gr05},
among many others.

\medskip

Set $p_d:=\frac{2(d+1)}{d+3}$. The optimal Stein--Tomas constant,
\[ 
{{\bf T}_d := {\bf T}_d(p_d), \text{ {{where}} }
  {\bf T}_d(p) := \sup_{0\neq f\in L^{p}(\R^d)} \frac{\left(\int_{\sph{d-1}}
  |\widehat{f}(\omega)|^2\d\sigma(\omega)\right)^{\frac12}}{\|f\|_{L^p(\R^d)}},} 
\] 
has attracted a great deal of attention in the recent literature.
In the lowest dimensional cases $d\in\{2,3\}$, the dual exponent $p_d'=\frac{2(d+1)}{d-1}$ is an even
integer, and the adjoint Stein--Tomas inequality can be reformulated in terms of multilinear convolution
operators. Following this path, Christ--Shao \cite{CS12} and Shao \cite{Sh16} established the precompactness
of maximizing sequences (modulo symmetries) for ${\bf T}_{d}$, and therefore the existence of maximizers,
when $d=3$ and $d=2$, respectively. The exact form of the maximizers for ${\bf T}_3$ was subsequently
determined by Foschi \cite{Fo15} via a remarkable geometric argument, but $d=3$ remains the unique dimension
for which such a characterization in known. In fact, even the mere existence of maximizers for ${\bf T}_d$ is
an outstanding open problem for every $d\geq 4$. Partial progress was recently {{obtained}} by
Frank--Lieb--Sabin \cite{FLS16}, who proved that if a well-known conjecture about the optimal constant in the Strichartz
inequality is true, then maximizers for ${\bf T}_d$ exist.
More precisely, the main theorem in \cite{FLS16} states the following: If Gaussians maximize the Strichartz
inequality for the Schr\"odinger equation in $\R^d$, then maximizing sequences for ${\bf T}_d$, normalized in
$L^{p_d}(\R^d)$, are precompact in $L^{p_d}(\R^d)$ up to translations and, in particular, maximizers for ${\bf T}_d$ exist.
It remains an open problem to turn this conditional result into an unconditional one. 

\subsection{Setting}
Given a subgroup $G\subset O(d)$ of the orthogonal group, a function $f:\R^d\to\Co$ is said to be {\it
$G$-symmetric in $\R^d$} if  $f\circ A=f$ holds for every $A\in G$.
An especially interesting situation arises when considering the subgroup $G_k:=O(d-k)\times O(k)$ 
for some $k\in\{0,1,\ldots,d\}$.
In this paper, we are interested in restriction estimates to the unit sphere,
\begin{equation}\label{eq:GkRestriction}
\left(\int_{\mathbb{S}^{d-1}} |\widehat{f}(\omega)|^q\, {\rm d}\sigma(\omega)\right)^{\frac1q}\leq C(k,d,p,q) \|f\|_{L^p(\R^d)},
\end{equation}
which hold in the class of $G_k$-symmetric functions. 
We are led to define the Banach space
\[L^{p}_{G_k}(\R^d):=\{f\in L^p(\R^d): f \text{ is } G_k\text{-symmetric} \}.\]
The cases $k\in\{0,d\}$ correspond to radial functions
on $\R^d$. If $f\in L^p(\R^d)$ is radial, then $\widehat{f}$ is continuous on $\R^d\setminus\{0\}$
whenever $1\leq p<\frac{2d}{d+1}$; see \cite[Prop.\@ 5.1]{SSIV}.
In particular, inequality \eqref{eq:GkRestriction} holds for radial functions provided $1\leq p<\frac{2d}{d+1}$
and $1\leq q\leq\infty$. This range of exponents is in fact optimal, since the radially symmetric 
counterexample from~\eqref{eq:IdNec} reveals that the adjoint of the restriction operator fails to be
bounded from $L^{q'}(\sph{d-1})$ to $L^{p'}(\R^d)$, for every $1\leq q\leq \infty$ and $p\geq
\frac{2d}{d+1}$. Thus the $L^p$--$L^q$ mapping properties of the restriction operator in the radial cases
$k\in\{0,d\}$ are completely understood. The cases $k\in\{1,d-1\}$ are likewise special.
Since Knapp's construction in $\R^d$ is rotationally invariant with respect to $d-1$ variables, {{we do not believe that $G_1$-symmetry or $G_{d-1}$-symmetry allows for a larger range of Lebesgue exponents on which Fourier restriction estimates can hold.}} 
In fact, in Remark~\ref{rem:G1Knapp} below we adapt  Knapp's construction to the $G_k$-symmetric setting, $k\in\{1,d-1\}$, thus revealing that no estimate beyond those predicted by the restriction conjecture is possible.
For this reason, we will focus on the situation when $k\in\{2,3,\ldots,d-2\}$.

\subsection{Results}

Our first result addresses the Stein--Tomas regime $q=2$.
 
\begin{theorem}\label{thm:main}
Let $d\geq 4, k\in\{2,3,\ldots,d-2\}$, and $m:=\min\{d-k,k\}$. 
Then the estimate
\begin{equation}\label{eq:NewTS}
\left(\int_{\mathbb{S}^{d-1}} |\widehat{f}(\omega)|^2\, {\rm d}\sigma(\omega)\right)^{\frac12}\leq C(k,d,p) \|f\|_{L^p(\R^d)}
\end{equation}
holds for every $G_k$-symmetric function $f:\R^d\to\Co$ if $1\leq p\leq \frac{2(d+m)}{d+m+2}$.
\end{theorem}
 
Given that $\frac{2(d+m)}{d+m+2}$ is strictly larger than the Stein--Tomas exponent $p_d=\frac{2(d+1)}{d+3}$, 
Theorem~\ref{thm:main} improves upon \eqref{eq:TSgen}. To the best of our knowledge, this result is new in
every dimension $d\geq 4$. 
\medskip

As an {immediate} application of Theorem~\ref{thm:main}, 
one may deduce improved mapping properties of the Helmholtz resolvent acting on $G_k$-symmetric functions.
More precisely, combining Theorem~\ref{thm:main} with {\cite[Theorem~1.2]{WY20}} for the {\it admissible extension triple} $(G,q,Q)=(G_k,(\frac{2(d+m)}{d+m+2})',{\mathbbm{1}})$ according to {\cite[Definition~1.1]{WY20}},
we obtain boundedness of the Helmholtz resolvent
$$ 
  \mathcal R:L^p_{G_k}(\R^d)\to L^{p'}_{G_k}(\R^d)
  {\text{ as long as }} \frac{2d}{d+2}\leq p<\frac{2d(d+m)}{d^2+(m+2)d+m-1}.
$$  
If $k\in\{2,3,\ldots,d-2\}$, then $m=\min\{d-k,k\}\geq 2$ and this upper bound for $p$ is strictly {larger} than $\frac{2(d+1)}{d+3}$,
which {in turn} is best possible in the non-symmetric setting \cite{KRS87}. 

Next, we apply our $G_k$-symmetric
Stein--Tomas inequality to establish the precompactness of maximizing sequences for the constrained
optimization problem 
\[
  {\bf T}_{d,k}(p) := \sup_{0\neq f\in L^p_{G_k}(\R^d)} \frac{\left(\int_{\sph{d-1}}
|\widehat{f}(\omega)|^2\d\sigma(\omega)\right)^{\frac12}}{\|f\|_{L^{p}(\R^d)}},
\]
and consequently the unconditional existence of maximizers for ${\bf T}_{d,k}(p)$. This is the content of
our second result.

\begin{theorem}\label{thm:G_kExistence}
Let $d\geq 4$, $k\in\{2,3,\ldots,d-2\}$, $m:=\min\{d-k,k\}$, and $1\leq p<\frac{2(d+m)}{d+m+2}$. 
Maximizing sequences for ${\bf T}_{d,k}(p)$, normalized in $L^p(\R^d)$, are precompact in
$L_{G_k}^p(\R^d)$.
In particular, maximizers for ${\bf T}_{d,k}(p)$ exist.
\end{theorem}

\noindent In particular, the set of all normalized maximizers for ${\bf T}_{d,k}(p)$ is itself compact {as long as $1\leq p<\frac{2(d+m)}{d+m+2}$}.  
Note that Theorem~\ref{thm:G_kExistence} applies to $p=p_d$, and thus {establishes the unconditional existence of} maximizers for
the classical Stein--Tomas inequality within the class of $G_k$-symmetric functions.
We remark that, in contrast to the conclusion of \cite[Theorem 1.1]{FLS16}, precompactness of complex-valued
maximizing sequences is {\it not} expected to hold modulo symmetries only, since $G_k$-symmetry eliminates
the loss of compactness due to translations.
There is still the danger that a maximizing sequence might {{conceivably}} converge weakly to zero.
To show that this is not the case, the proof of Theorem \ref{thm:G_kExistence} will make use of a decay
property of the Fourier transform which is only available in the $G_k$-symmetric setting for $2\leq k\leq
d-2$; see Proposition \ref{prop:decay} and Corollary \ref{cor:decay} below.
{The endpoint case of Theorem \ref{thm:G_kExistence} remains an interesting open question.}

\medskip 

The range of exponents of Theorem~\ref{thm:main} turns out to be optimal.
This is a consequence of our third result, which exhibits the necessary conditions for the
 restriction estimate~\eqref{eq:GkRestriction} to hold within the class of $G_k$-symmetric functions.

\begin{theorem}\label{thm:Nec}
Let $d\geq 4$,  $k\in\{2,3,\ldots,d-2\}$, $m:=\min\{d-k,k\}$, and $1\leq p,q\leq \infty$. 
If inequality~\eqref{eq:GkRestriction} holds within the class of $G_k$-symmetric functions, then one of
the following conditions is satisfied:
\begin{itemize}
\item[(i)] $\frac{d+1}{2d}<\frac1p<\frac{m+1}{2m}$ and $\frac{d+m}{p}+\frac{d-m}{q}\geq d+1;$
\item[(ii)] $\frac1p=\frac{m+1}{2m}$ and $\frac1p+\frac1q> 1;$ 
\item[(iii)] $\frac1p>\frac{m+1}{2m}$ and $\frac1p+\frac1q\geq 1$.
\end{itemize}
\end{theorem}

 We shall see in Corollary~\ref{cor:GkRiesz} that the necessary conditions from (i)  are sufficient when
 $\frac{d+m+2}{2(d+m)}\leq\frac1p<\frac{m+1}{2m}$.
For larger values of $p$, the
sufficiency of these conditions remains an open problem. 
On the other hand, conditions (ii) and (iii) turn out
to be sufficient for estimate~\eqref{eq:GkRestriction} to hold in the $G_k$-symmetric setting.
The crux of the matter is a {{(mixed)}} Lorentz space estimate at the endpoint
$p={q'}=\frac{2m}{m+1}$, which is the content of our fourth result. {{If $m:=\min\{d-k,k\}<\frac d2$ we set $X_p:= L^{p,1}(\R^d)$, otherwise $m=\frac d2$ and let
\begin{equation}\label{eq_Xp}
 X_p :=L_x^{p,1}(\R^{d-k}; L^{p,1}_y(\R^k))+L_y^{p,1}(\R^{k}; L^{p,1}_x(\R^{d-k})).
 \end{equation}}}

\begin{theorem}\label{thm:easy}
Let $d\geq 4, k\in\{2,3,\ldots,d-2\}$, and $m:=\min\{d-k,k\}$.
Then the estimate
\[\|\widehat{f}\|_{L^{\frac{2m}{m-1},\infty}(\sph{d-1})}\leq C(k,d)\|f\|_{{{X_{\frac{2m}{m+1}}}}}\]
holds for every $G_k$-symmetric function $f:\R^d\to\Co$.
\end{theorem}

Compactness of $\sph{d-1}$ and real interpolation {{of Lorentz \cite[\S~5.3]{BL} and mixed Lorentz spaces \cite[Cor.\@ 1]{Ma23}}} between Theorem~\ref{thm:main}, Theorem~\ref{thm:easy}, and
the trivial endpoint $(p,q)=(1,\infty)$ together imply a range of estimates which we now record;
see Figure~\ref{fig:Riesz} for the corresponding Riesz diagram.

\begin{corollary}\label{cor:GkRiesz}
  Let $d\geq 4, k\in\{2,3,\ldots,d-2\}$, and $m:=\min\{d-k,k\}$. Then inequality~\eqref{eq:GkRestriction}
holds in the class of $G_k$-symmetric functions if one of the following conditions is satisfied$:$
  \begin{itemize}
	\item[(i)] $\frac{d+m+2}{2(d+m)}\leq \frac1p<\frac{m+1}{2m} $ and $\frac{d+m}{p}+\frac{d-m}{q}\geq d+1;$
	\item[(ii)] $\frac1p=\frac{m+1}{2m}$ and $\frac1p+\frac1q> 1;$ 
	\item[(iii)] $\frac1p>\frac{m+1}{2m}$ and $\frac1p+\frac1q\geq 1$.
  \end{itemize}
\end{corollary}
As a concluding remark, note that Corollary \ref{cor:GkRiesz} (i) implies the diagonal estimate
\begin{equation}\label{eq:PreGuth}
\|\widehat{f}\|_{L^p(\sph{d-1})}\lesssim \|f\|_{L^p(\R^d)}
\end{equation}
 at $p=\frac{2(d+m)}{d+m+2}$, for every $G_k$-symmetric $f:\R^d\to\Co$.
In this case, if $m=k=\lfloor\frac d2\rfloor$, then $p'
=2+\frac83d^{-1}+O(d^{-2})$.
That \eqref{eq:PreGuth} holds for {\it general} $f\in L^p(\R^d)$ in the range $p'>2+\frac83d^{-1}+O(d^{-2})$ was recently proved by Guth \cite{Gu18}, and later improved in \cite{HR19}.
This remains to date the state of the art in the high-dimensional restriction conjecture.

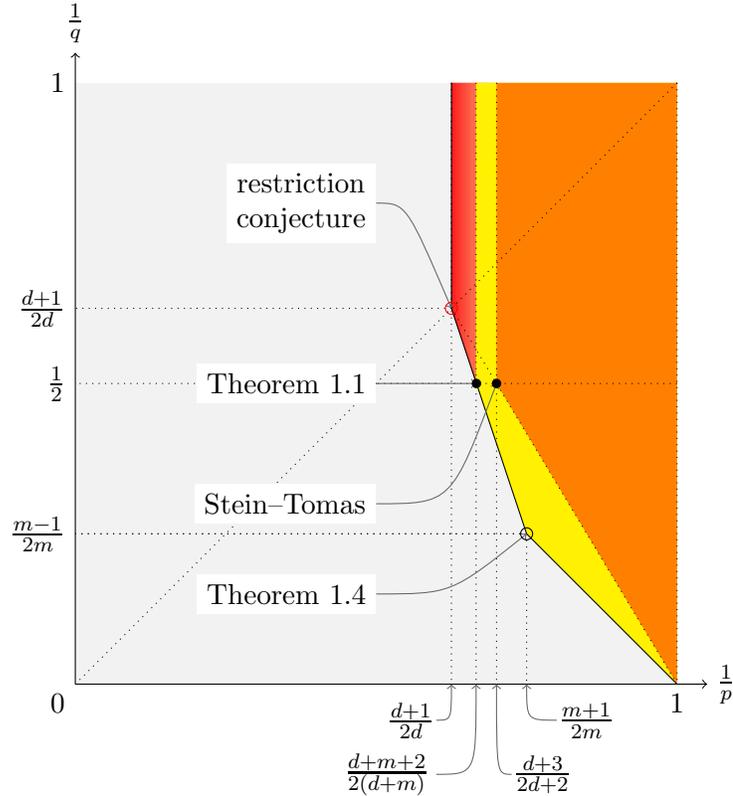
\begin{figure}[htbp]
  \centering
  \begin{tikzpicture} [scale = 8]
    \fill[black!5] (0, 0) rectangle (1, 1);
    \shade[left color = red, right color = orange!50] (0.625, 1) -- (0.625, 0.625) -- (0.666, 0.5) -- (0.7, 1) -- cycle;
    \fill[yellow] (0.75, 0.25) -- (0.666, 0.5) -- (0.666,1) -- (1, 1) -- (1, 0) -- cycle;
    \fill[orange] (0.7, 0.5) -- (0.7, 1)  -- (1, 1) -- (1, 0) -- cycle;
    \draw[->] (0, 0) -- (1.05, 0) node[right] {$\frac1p$};
    \draw[->] (0, 0) -- (0, 1.05) node[above] {$\frac1q$};
    \draw (0, 1) node[left] {$1$};
    \draw[dotted] (1, 1) -- (0, 0) node[below left] {$0$}; 
    \draw[dotted] (1, 1) -- (1, 0) node[below] {$1$};
    \draw[dotted] (0.625, 1) -- (0.625, 0);
    \draw[dotted] (0.666, 1) -- (0.666, 0);
    \draw[dotted] (0.75, 0.25) -- (0.75, 0);
    \draw[dotted] (0.75, 0.25) -- (0, 0.25);
    \draw[dotted] (0.625, .625) -- (1, 0);
    \draw (0.75, 0.25) -- (1, 0);
    \draw (0.75, 0.25) -- (0.625, 0.625);
    \draw (0.625, 0.625) -- (0.625,1);
    \draw[dotted] (0.75, 0.25) -- (0, 0.25) node[left] {$\frac{m-1}{2m}$};
    \draw[<-, black!60, text = black] (0.625, 0) .. controls (0.625, -0.06) .. (0.6, -0.06) node[left = -3] {$\frac{d+1}{2d}$};
    \draw[dotted] (0.7, 1) -- (0.7, 0);
    \draw[<-, black!60, text = black] (0.7, 0) .. controls (0.7, -0.15) .. (0.725, -0.15) node[right = -3] {$\frac{d+3}{2d+2}$};
    \draw[<-, black!60, text = black] (0.75, 0) .. controls (0.75, -0.06) .. (0.8, -0.06) node[right = -3] {$\frac{m+1}{2m}$};
    \draw[dotted] (1, 0.5) -- (0, 0.5) node[left] {$\frac12$};
    \draw[dotted] (0.625, 0.625) -- (0, 0.625) node[left] {$\frac{d+1}{2d}$};
    \draw[<-, black!60, text = black] (0.7, 0.5) .. controls (0.625, 0.3) .. (0.5, 0.3) node[fill = white, left] {Stein--Tomas};
    \draw[<-, black!60, text = black] (0.666, 0.5) .. controls (0.625, 0.5) .. (0.5, 0.5) node[fill = white, left] {Theorem \ref{thm:main}};
    \draw[<-, black!60, text = black] (0.75, 0.25) .. controls (0.625, 0.15) .. (0.5, 0.15) node[fill = white, left] {Theorem \ref{thm:easy}};
\draw[black,fill=black] (0.7, 0.5) circle [radius = 0.007];
    \draw[black,fill=black] (0.6666, 0.5) circle [radius = 0.007];
        \draw[<-, black!60, text = black] (0.666, 0) .. controls (0.666, -0.15) .. (0.6, -0.15) node[left = -1] {$\frac{d+m+2}{2(d+m)}$};
    \draw[<-, black!60, text = black] (0.625, 0.625) .. controls (0.55, 0.8) .. (0.5, 0.8) node[fill = white, left, align=center] {restriction\\conjecture};
    \draw[black] (0.75, 0.25) circle [radius = 0.01];
    \draw[red] (0.625, 0.625) circle [radius = 0.01];
  \end{tikzpicture}
  \caption{Riesz diagram for the $G_k$-symmetric restriction problem to $\sph{d-1}$. 
  Estimates in the orange region follow from the Stein--Tomas inequality, 
  estimates in the yellow region follow from 
  Corollary \ref{cor:GkRiesz} and, in light of Theorem \ref{thm:Nec},
  no estimates are possible within the grey region.
  The possibility of estimates in the red region remains an open problem.
  }
  \label{fig:Riesz}
\end{figure}

  \subsection{Historical remarks}
The expectation that further Fourier restriction estimates are available within certain classes of functions exhibiting additional symmetries has been extensively explored in the literature.
For instance, the well-known fact that the restriction conjecture holds for radial functions has been generalized to the class of products of radial functions and spherical harmonics \cite{DG04}.
On the  paraboloid and the cone, the restriction conjecture has been established  for functions which are invariant under spatial rotations, and further estimates are known to hold in the cylindrically symmetric case, 
but only for dyadically supported functions; see \cite{Sh09,Sh09b}.
This was later generalized to the mixed norm setting; see \cite{MZZ12,MZZ19}.
Very recently, the $G_k$-symmetric setting has been proposed in \cite{WY20}, albeit in the context of
adjoint restriction estimates on the unit sphere involving weight functions supported on sets of the form
$\{(y,z)\in\R^{d-k}\times\R^k:
|y| \leq C|z|^{-\alpha}\}$ for some $C, \alpha>0$. 
With the present work, we also aim to initiate a systematic exploration of the more general $O(k_1)\times\ldots\times O(k_n)$-symmetric restriction problem to $\sph{d-1}$, with $\sum_{j=1}^n k_j=d$. This so-called {\it block-radial} symmetry has been extensively explored in the related context of Sobolev space embeddings, starting with the pioneering work of Lions \cite{Li82}.

In all of the above problems, it is very natural to ask about maximizers and optimal constants.
{\it Sharp restriction theory} is a vibrant area of research which has flourished in the last decade.
A natural first step towards sharp restriction inequalities is to establish the existence of maximizers. 
This provides a stepping stone towards a qualitative analysis, the discovery of explicit maximizers, and the corresponding full characterization 
(which up to now is only available in very special circumstances).
Works addressing the existence of maximizers for inequalities of endpoint Fourier
restriction type tend to be a {\it tour de force} in classical analysis, using a variety
of sophisticated techniques. 
Besides the aforementioned precompactness results on the unit sphere \cite{CS12, FLS16, Sh16}, we highlight the general method developed in \cite{FVV11} together with the (unconditional) existence results on the paraboloid \cite{BV07,St20}, the cone \cite{Qu13,Ra12}, the hyperboloid \cite{COSS19,COSSS21,Qu15}, and the moment curve \cite{BS20}.
For a more comprehensive discussion and further references, we refer the interested reader to the survey \cite{FOS17}.

\subsection{Structure of the paper}
In  \S~\ref{sec:Prelim}, we discuss some 
analytic preliminaries about
the interplay between $G_k$-symmetry and the Fourier transform. We also investigate a useful
family of oscillatory integrals, and establish  weighted versions of the classical inequalities of Hausdorff--Young and Hardy--Littlewood--Sobolev.
 In \S~\ref{sec:WRtypeE}, we prove the weighted restriction estimates that will play a central role in the proof
of Theorem \ref{thm:main}, which is then the subject of  \S~\ref{sec:ProofThm1}.
Theorems \ref{thm:G_kExistence}, \ref{thm:Nec}, \ref{thm:easy} are proved in \S~\ref{sec:ProofThmExistence}, \S~\ref{sec:ProofThm2}, \S~\ref{sec:ProofThm3}, respectively.

\subsection{Forthcoming notation}
We reserve the letter $d$ to denote the dimension of the ambient space $\R^d$.
Given a Lebesgue exponent $p\in[1,\infty]$, its dual is $p'=p/(p-1)$.
The usual Lebesgue and Lorentz spaces are denoted by $L^p(\R^d)$ and $L^{p,s}(\R^d)$,
respectively, and the corresponding (quasi-)norms are indexed accordingly. 
The Schwartz space is denoted by $\mathcal S(\R^d)$.
The Fourier transform on $\R^d$ is normalized in the following way:
\[
  \widehat{f}(\xi)=\int_{\R^d} f(x) e^{-i\xi\cdot x}\d x.
\]
The indicator function of a set $E\subset\R^d$ is denoted by $\mathbbm{1}_E$, and its Lebesgue measure by $|E|$. 
The usual surface measure on $\sph{d-1}$ is denoted by $\sigma$, and its surface area is given by
$\sigma(\sph{d-1})=\int_{\sph{d-1}}\d\sigma(\omega)={2\pi^{\frac d2}}\Gamma(\frac d2)^{-1}$.
We shall write 
\[\|F\|_{L^p(\sph{d-1})}^p=\int_{\sph{d-1}}|F(\omega)|^p\d\sigma(\omega).\]
Finally, we use the shorthand notation $X\lesssim Y$, $Y\gtrsim X$, $X=O(Y)$ to denote the estimate $|X|\leq CY$ for
some positive constant $C$ which is only allowed to depend on the space dimension $d$, the symmetry index $k$, and possibly some other {\it fixed} parameters.  
We also write $X\simeq Y$ for $X\lesssim Y \lesssim X$.
 
 
\section{Preliminaries}\label{sec:Prelim}
In this section, we discuss some analytic preliminaries related to
duality, Bessel functions, $G_k$-symmetry, oscillatory integrals, and weighted variants of the classical inequalities of Hausdorff--Young and
Hardy--Littlewood--Sobolev.

\subsection{Duality}
The adjoint of the restriction operator to the unit sphere, $\mathcal R:L^p(\R^d)\to L^q(\sph{d-1}), f\mapsto
\widehat{f}\mid_{\sph{d-1}}$,  is the {\it extension operator}, $\mathcal R^\ast:L^{q'}(\sph{d-1})\to
L^{p'}(\R^d)$, $F\mapsto \widehat{F\sigma}$, defined at $x\in\R^d$ via the expression
\begin{equation}\label{eq:ExtensionOp}
  \widehat{F\sigma}(x):=\int_{\sph{d-1}} F(\omega) e^{i x\cdot\omega} \d\sigma(\omega).
\end{equation}
In particular, if $F\equiv 1$, then 
\begin{equation}\label{eq:SigmaHat}
\widehat{\sigma}(x)=(2\pi)^{\frac d2} |x|^{\frac {2-d}2}J_{\frac{d-2}2}(|x|),
\end{equation}
where $J_\nu$ denotes the Bessel function of the first kind; this is a special case of the so-called
Bochner--Hecke formula (see \cite[p.~347]{St93}).
From the classical asymptotic formulae for Bessel functions, see \eqref{eq:BesselAsymp}--\eqref{eq:BesselAsympNear0} below, or via a direct stationary phase argument, one has  that
\begin{equation}\label{eq:AsymptSigmaHat}
|\widehat{\sigma}(x)|=O( (1+|x|)^{\frac{1-d}2});
\end{equation}
see \cite[p.~348]{St93}. Estimate \eqref{eq:AsymptSigmaHat} is a manifestation of the well-known fact that curvature causes the Fourier transform to decay.
\medskip

In this dual setting, a function $F:\sph{d-1}\to\Co$ is said to be {\it $G_k$-symmetric on $\sph{d-1}$} if $F\circ A=F$, for every $A\in G_k$. In particular, a set $S\subset\sph{d-1}$ will be called $G_k$-symmetric if its
indicator function $\mathbbm{1}_S$ is $G_k$-symmetric, and similarly for subsets $E\subset\R^d$.

\subsection{Bessel functions}
In view of identity \eqref{eq:SigmaHat},  the Bessel function
\[J_\nu(r):=\sum_{j=0}^\infty \frac{(-1)^j (\frac12 r)^{2j+\nu}}{j!\Gamma(\nu+j+1)}\]
is expected to play a role
in the analysis. Only $\nu,r\geq 0$ will be of interest. 
As is well-known, for any fixed $\nu\geq 0$, one has that
\begin{equation}\label{eq:BesselAsymp}
J_\nu(r)=
\left(\frac{\pi r}2\right)^{-\frac12}\cos\left(r-\frac{\nu\pi}2-\frac{\pi}4\right)+O(r^{-\frac32}),
\quad\text{as }r\to\infty;
\end{equation}
\begin{equation} \label{eq:BesselAsympNear0}
 |J_\nu(r)|\leq\frac{r^\nu}{2^\nu \Gamma(\nu+1)},
  \quad\text{for all }r\geq 0;
\end{equation}
see \cite[pp.~356--357]{St93} and \cite[pp.~48--49]{Wa66}. 
From~\eqref{eq:BesselAsymp}--\eqref{eq:BesselAsympNear0}, it is natural to expect the following result.

\begin{lemma}\label{lem:HardBessel}
Let $\nu\geq 0$.
There exist a constant $A_\nu\in\Co\setminus\{0\}$ and a function $R_\nu:(0,\infty)\to \R$, such
that
\begin{equation}\label{eq:BesselIdAR}
J_{\nu}(r)= (A_\nu e^{ir}+\overline{A_\nu} e^{-ir})r^{-\frac12}\mathbbm{1}_{[1,\infty)}(r)+R_\nu(r),
\end{equation}
 where additionally $|R_\nu(r)|\lesssim  r^\nu(1+r)^{-\nu-\frac32}$, for every $r\geq 0$. 
\end{lemma}
\begin{proof}
Let $A_\nu:=(\frac2{\pi})^{\frac12} e^{-i(\frac{\nu\pi}2+\frac{\pi}4)}$, and define the function $R_\nu$ via
identity \eqref{eq:BesselIdAR} above.
Then the desired estimate for $R_\nu$ follows from \cite[p.~201]{Wa66}.
\end{proof}

\subsection{$G_k$-symmetry}
Given a $G_k$-symmetric function $f:\R^d\to\Co$, we shall define $f_0:(0,\infty)^2\to\Co$ via $f_0(|y|,|z|):=f(y,z)$, for $(y,z)\in \R^{d-k}\times\R^k$, and denote the corresponding
Fourier variables by $(\eta,\zeta)\in \R^{d-k}\times\R^k$.
\medskip

A function is radial if and only if its Fourier transform is radial.
This well-known fact admits the following straightforward generalization:
$G_k$-symmetry is preserved under the Fourier transform.
Indeed, identity \eqref{eq:SigmaHat} implies the following result.

\begin{lemma}\label{lem:GkSymFT}
Let $f\in \mathcal S(\R^d)$ be $G_k$-symmetric, and set $f_0(|y|,|z|):=f(y,z)$. Then the following identity holds at every  $(\eta,\zeta)\in \R^{d-k}\times\R^k:$
\[\widehat{f}(\eta,\zeta)=(2\pi)^{\frac d2} |\eta|^{\frac{2-d+k}{2}} |\zeta|^{\frac{2-k}2} \int_0^\infty \int_0^\infty \rho_1^{\frac{d-k}2}\rho_2^{\frac k2}  f_0(\rho_1,\rho_2) J_{\frac{d-k-2}{2}}(\rho_1|\eta|)J_{\frac{k-2}2}(\rho_2|\zeta|)\d\rho_1\d\rho_2.\]
\end{lemma}
\begin{proof}
 This follows from an explicit computation in polar coordinates 
  and the $G_k$-invariance of $f$.
Indeed, introducing coordinates $(\rho_1,\omega_1)$ in $\R^{d-k}$ and $(\rho_2,\omega_2)$ in $\R^k$, we find
   that
   \begin{align*}
    \widehat f(\eta,\zeta)=&\int_0^\infty \int_0^\infty  \rho_1^{d-k-1}\rho_2^{k-1}  f_0(\rho_1,\rho_2) \\
    &\times \left(\int_{\sph{d-k-1}} e^{-i\omega_1\cdot\rho_1\eta}\d\sigma(\omega_1)  \right)
    \left(\int_{\sph{k-1}}e^{-i\omega_2\cdot\rho_2\zeta} \d\sigma(\omega_2)\right)  
     \d\rho_2\d\rho_1.
  \end{align*}
  The antipodal
  change of variables $(\omega_1,\omega_2)\rightsquigarrow-(\omega_1,\omega_2)$ then reduces the claim to identity
  \eqref{eq:SigmaHat}. The proof is complete.
\end{proof}

We will need to integrate $G_k$-symmetric functions in $\R^d$ over the unit sphere. 
The next result provides the corresponding formula.

\begin{lemma}\label{lem:SphInt}
Let $f:\R^d\to\Co$ be $G_k$-symmetric and integrable on $\sph{d-1}$, and
set $f_0(|\eta|,|\zeta|):=f(\eta,\zeta)$. Then the following identity holds:
\[
  \int_{\sph{d-1}} f(\eta,\zeta)\d\sigma(\eta,\zeta)
  =\frac{\sigma(\sph{d-k-1})\sigma(\sph{k-1})}{\sigma(\sph{d-1})} \int_0^1 r^{d-k-1}(1-r^2)^{\frac{k-2}2}
f_0(r,\sqrt{1-r^2}) \d r.
\]
\end{lemma}
\begin{proof}
This follows from slice integration \cite[Theorem A.4]{ABR01} and the $G_k$-invariance of $f$.
\end{proof}

As far as pointwise bounds for the extension operator of a $G_k$-symmetric function on $\sph{d-1}$ are
concerned, we have the following result.
\begin{proposition}\label{prop:decay}
There exists $C=C_{k,d}<\infty$ 
such that pointwise bound
\begin{equation}\label{eq:Pointwise}
|\widehat{F\sigma}(y,z)|\leq C_{k,d} \|F\|_{L^2(\sph{d-1})} (1+|y|)^{\frac{k+1-d}2} (1+|z|)^{\frac{1-k}2} 
\end{equation}
holds
for every $(y,z)\in\R^{d-k}\times\R^k$ and 
every $G_k$-symmetric $F\in L^2(\mathbb S^{d-1})$. 
\end{proposition}

\begin{proof}
Let $F\in L^2(\sph{d-1})$ be $G_k$-symmetric and given.
Start by noting that the function $\widehat{F\sigma}$ is real-analytic since it is the Fourier transform
of a compactly supported distribution.
Set $F_0(r):=F(r\omega,\sqrt{1-r^2}\nu)$ for $r\in[0,1], \omega\in\sph{d-k-1},\nu\in\sph{k-1}$; see \cite[p.~241]{ABR01}.
Since $F$ is 
integrable on $\sph{d-1}$,  we can appeal to the slice integration formula from \cite[Theorem A.4]{ABR01} to conclude,  via a passage to polar coordinates, that
\[ 
  \widehat{F\sigma}(y,z)\simeq \int_0^1 r^{d-k-1}(1-r^2)^{\frac{k-2}2} F_0(r) \left(\int_{\sph{d-k-1}}e^{i
  y\cdot r\omega}\d\sigma(\omega)\right)\left(\int_{\sph{k-1}}e^{iz\cdot \sqrt{1-r^2}\nu}\d\sigma(\nu)\right)
  \d r.\] 
Note that the implicit constant depends only on $k,d$.
The Cauchy--Schwarz inequality, Lemma \ref{lem:SphInt}, 
 and estimate \eqref{eq:AsymptSigmaHat} together imply
 \begin{align*}
 |\widehat{F\sigma}(y,z)|^2\lesssim\|F\|^2_{L^2(\sph{d-1})}   \int_0^1 r^{d-k-1}(1-r^2)^{\frac{k-2}2} (1+r|y|)^{k+1-d} (1+\sqrt{1-r^2}|z|)^{1-k}\d r. 
 \end{align*}
The pointwise bound \eqref{eq:Pointwise} follows from this via another application of the Cauchy--Schwarz inequality together with elementary considerations in both regimes $|y|\leq 1,|y|>1$ separately, and similarly for $z$.
This concludes the proof of the proposition.
\end{proof}

For the purposes of the upcoming analysis in \S~\ref{sec:ProofThmExistence}, we will be interested in the following straightforward consequence of Proposition \ref{prop:decay} which, in the language of concentration compactness theory \cite{Li84}, will preclude {\it vanishing}  (i.e.\@ mass sent to infinity) of maximizing sequences.

\begin{corollary}\label{cor:decay}
Let $d\geq 4$ and $k\in \{2,3,\ldots,d-2\}$.
Then, for every $\varepsilon>0$, there exists
$R=R(k,d,\varepsilon)<\infty$ for which $|\widehat{F\sigma}(x)|<\varepsilon$ if $|x|>R$, for every $G_k$-symmetric $F\in L^2(\sph{d-1})$ such that $\|F\|_{L^2(\sph{d-1})}=1$.
\end{corollary}
\begin{remark}
\emph{That no such property can hold for general $F\in L^2(\sph{d-1})$ follows at once from the fact that the extension operator intertwines modulation and translation:
\[\mathcal R^\ast{(e^{iy\cdot}F)}(x)
=\int_{\sph{d-1}} e^{iy\cdot \omega}F(\omega) e^{ix\cdot\omega}\d\sigma(\omega)
=\int_{\sph{d-1}} F(\omega) e^{i(x+y)\cdot\omega}\d\sigma(\omega)
=\mathcal R^\ast(F)(x+y).\]
Indeed, if a nonzero function $F$ and its modulation $e^{iy\cdot} F$ are both $G_k$-symmetric on $\sph{d-1}$, then necessarily $y=0$.}
\end{remark}

\subsection{Oscillatory integrals}
We will use the following simple bound on a certain class of oscillatory integrals.

\begin{lemma}\label{prop:OscillatoryIntegral}
  Let $0<\gamma\neq 1$. Then there exists a constant $C=C_\gamma<\infty$ such that, for every $a\geq 1$ and $\lambda\in
  [-2,2]\setminus\{0\}$, the following inequality holds:
  \[
    \left| \int_a^\infty r^{-\gamma} e^{i\lambda r} \d r\right|
    \leq  C\begin{cases}
      |\lambda|^{\gamma-1}  &\text{if } 0<\gamma< 1, \\
      a^{1-\gamma}   &\text{if } \gamma>1.
    \end{cases}
  \]
\end{lemma}
\begin{proof}
  No generality is lost in assuming that $\lambda\in(0,2]$. Changing variables $\lambda r=\rho$,
  \[
    \int_a^\infty r^{-\gamma} e^{i\lambda r} \d r
    = \lambda^{\gamma-1} \int_{\lambda a}^\infty \rho^{-\gamma} e^{ i\rho} \d\rho,
  \]
  we see that the desired conclusion follows from
  \[
    \sup_{R>0} \left| \int_R^\infty \rho^{-\gamma} e^{i\rho} \d\rho\right|
    < \infty \quad\text{if } 0<\gamma<1,\qquad  
    \sup_{R>0}\,  R^{\gamma-1} \left| \int_R^\infty \rho^{-\gamma} e^{i\rho} \d\rho\right|
    < \infty \quad\text{if } \gamma>1.
  \]
  The first estimate follows from integration by parts,
  and the second estimate follows even more simply from an application of the triangle inequality.
\end{proof}

\subsection{Weighted Hausdorff--Young Inequality}
While Lemma~\ref{lem:weightedHY} below is clearly related to Pitt's inequality (also known as Hardy's
inequality; see~\cite{Be08, Be12}), we choose to present a self-contained, short proof of the special
one-dimensional case which will be directly relevant to our analysis.
For convenience, set $\mathcal F(f):=\widehat f$.

\begin{lemma} \label{lem:weightedHY} 
Let $1<p\leq 2\leq q<\infty$, and $\delta:=1-\frac1p-\frac1q$. 
If $0\leq\delta<1$, then
 the estimate 
\[  \|\mathcal F (f |\cdot|^{-\delta})\|_{L^q(\R)} \leq C(p,q) \| f \|_{L^p(\R)}\]
holds for every $f\in L^p(\R)$.
\end{lemma}
 \begin{proof}
 If $\delta=0$, then the result amounts to the Hausdorff--Young inequality.
 If $\delta\in(0,1)$, then we have that 
 \[\|\mathcal F (f |\cdot|^{-\delta})\|_{L^q(\R)}
 =\|  \widehat f\ast \mathcal F(|\cdot|^{-\delta})\|_{L^q(\R)}
 \simeq \|  \widehat f\ast (|\cdot|^{-(1-\delta)})\|_{L^q(\R)}
 \lesssim \|\widehat f\|_{L^{p'}(\R)}
 \lesssim \|f\|_{L^p(\R)}.\]
 This chain of estimates results from consecutive applications of the Hardy--Littlewood--Sobolev
  \cite[p.~354]{St93}  and the Hausdorff--Young inequalities.
 \end{proof}

\subsection{Weighted Hardy--Littlewood--Sobolev Inequality}

Our analysis will rely on the $L^p$--$L^q$ mapping properties of the
following family of integral operators, indexed by $a,b\in\R$ and acting on functions $f:\R_+:=[0,\infty)\to\Co$ via
\begin{align*}
\mathcal T_{a,b}(f)(x)&:=x^{-a}\int_{y\leq x} y^{-b} f(y)\d y.  
\end{align*}

\begin{lemma}\label{lem:MappingT}
Let $1< p\leq q<\infty$ and $a,b\in\R$.
Then
$\mathcal T_{a,b}:L^p(\R_+)\to L^q(\R_+)$ is bounded if $bp'<1$ and $\frac1{p'}+\frac1{q}=a+b$.
\end{lemma}
\begin{proof}
For any $x\geq 0$, from H\"older's inequality it follows that
 \[|\mathcal T_{a,b}(f)(x)|\leq x^{-a}\left(\int_0^x y^{-bp'}\d y\right)^{\frac1{p'}}
 \|f\|_{L^p(\R_+)}\simeq x^{\frac1{p'}-a-b}\|f\|_{L^p(\R_+)},\] 
 where the implicit constant is finite as long as $bp'<1$.
  This implies 
\begin{equation*}
\|\mathcal T_{a,b}(f)\|_{L^{q,\infty}(\R_+)}\lesssim \left\|(\cdot)^{\frac1{p'}-a-b}\right\|_{L^{q,\infty}(\R_+)}\|f\|_{L^p(\R_+)},
\end{equation*}
where the first term on the right-hand side is finite precisely when $\frac1{p'}+\frac1{q}=a+b$. 
To conclude, note that 
the claimed strong-type estimate for $1<p\leq q<\infty$ follows from the Marcinkiewicz interpolation theorem \cite[p.~9]{BL}
applied to the bounds $L^p(\R_+)\to L^{q,\infty}(\R_+)$ 
with $(\frac{1}{p},\frac{1}{q})\in\{(1-\max\{b,0\}-\delta, a+\min\{b,0\}-\delta),(1-a-b+\delta,\delta)\}$, for sufficiently small
enough $\delta>0$. \
Indeed, all exponents $1<p\leq q<\infty$ satisfying $bp'<1$ and $\frac1{p'}+\frac1{q}=a+b$ are covered by this interpolation procedure since $0<a+b\leq 1$.
This finishes the proof of the lemma.
\end{proof}

Given $\ell\in(0,\infty)$ and certain $a,b\in\R$, the need will arise to consider the following related family of integral operators, acting on functions $f:[0,\ell]\to\Co$ via:
\begin{align}
\mathcal S_{a,b}(f)(x)&:= x^{-a}\int_{y\leq x}  (x-y)^{-b} f(y)\d y.\label{eq:defOpS}
\end{align}
Useful\footnote{Albeit non-optimal. We omit trivial improvements of Lemma \ref{lem:MappingS} (obtainable e.g.\@ via H\"older's inequality) which are not directly relevant to the forthcoming analysis.} 
estimates for $\mathcal S_{a,b}$ follow from the Stein--Weiss inequality \cite{SW58},
which extends the Hardy--Littlewood--Sobolev inequality for fractional integrals.

\begin{lemma}\label{lem:MappingS}
Let $\ell\in(0,\infty), 1<p\leq q<\infty$, $a\geq 0$, and $0<b<1$.
Then $\mathcal S_{a,b}:L^p([0,\ell])\to L^q([0,\ell])$ is bounded if   
$\frac1{p'}+\frac1{q}\geq a+b$.
\end{lemma}

\begin{remark}
\emph{In the endpoint case $\frac1{p'}+\frac1{q}=a+b$, the assumption $p\leq q$ is 
in fact
necessary for the $L^p$--$L^q$ boundedness of $\mathcal S_{a,b}$.
Indeed, if $p>q$, then let $0<\varepsilon<\frac{p}q-1$ and
\[f_\varepsilon(x):=x^{-\frac1p} |\log(x)|^{-\frac{1+\varepsilon}p}\mathbbm1_{[0,\frac12]}(x).\]
It is easy to check that $f_\varepsilon\in L^p([0,1])$, but that $|S_{a,b}(f_\varepsilon)(x)|\gtrsim x^{\frac1{p'}-a-b} |\log(x)|^{-\frac{1+\varepsilon}p}$ for all sufficiently small $x>0$. In particular, $S_{a,b}(f_\varepsilon)\notin L^q([0,1])$ in view of our choice of $\varepsilon$.}
\end{remark}
\begin{proof}[Proof of Lemma \ref{lem:MappingS}]
For every $0\leq x\leq\ell$, we have 
\[|\mathcal S_{a,b}(f)(x)|
  \leq  |x|^{-\delta a}\int_{\R}  |x-y|^{-b} |y|^{-(1-\delta)a} |(f\, \mathbbm1_{[0,\ell]})(y)|\d y,\]
as long as $0<\delta<1$.
The hypotheses make it possible to choose $\delta\in(0,1)$ and exponents $\tilde p,\tilde q$ satisfying $1<\tilde p\leq p\leq q\leq \tilde q<\infty$, in such a way that 
 \[\frac{1}{\tilde p'}>(1-\delta)a,\,\,\,\frac{1}{\tilde q}>\delta a,\,\,\, \frac{1}{\tilde p'}+\frac{1}{\tilde q}=\delta a+b+(1-\delta)a.\] 
From the Stein--Weiss inequality \cite[Theorem $B_1^\ast$]{SW58}, it then follows that 
$\|\mathcal S_{a,b}(f)\|_{L^{\tilde q}(\R)} \lesssim \|f\,\mathbbm1_{[0,\ell]}\|_{L^{\tilde p}(\R)}$,
which implies the desired conclusion via the inclusion of Lebesgue spaces on bounded intervals since
$\tilde p\leq p$ and $q\leq \tilde q$.
This concludes the proof of the lemma. 
\end{proof}

In the course of the proof of Proposition \ref{prop:GenSB} below, we will also invoke bounds for the adjoint of the operator $\mathcal S_{a,b}$, whose form we record here:
\begin{equation}\label{eq:AdjointS}
  \mathcal S_{a,b}^*(g)(y):= \int_{x\geq y} (x-y)^{-b} x^{-a} g(x)\d x.
\end{equation}
Naturally, $\mathcal S_{a,b}^*:L^{q'}([0,\ell])\to L^{p'}([0,\ell])$ if and only if 
$\mathcal S_{a,b}:L^{p}([0,\ell])\to L^{q}([0,\ell])$.


\section{Weighted 2D Restriction-type estimates} \label{sec:WRtypeE}

In this section, we analyze the $L^p(\R^2)$--$L^q(\sph{1})$ mapping properties of the operator 
$\mathcal R_{\alpha,\beta}$, defined as follows:
\begin{equation}\label{eq:DefRab}
  \mathcal R_{\alpha,\beta}(g)(\omega) 
  := \int_{\R^2}  g(x) \mathbbm{1}_{|\omega_1||x_1|\geq  1}\mathbbm{1}_{|\omega_2||x_2|\geq 1}
  (1+|x_1|)^{-\alpha}(1+|x_2|)^{-\beta} e^{-ix\cdot(|\omega_1|,|\omega_2|)} \d x;
\end{equation}
here $\alpha,\beta>0, \omega=(\omega_1,\omega_2)\in\sph{1}$, and $x=(x_1,x_2)\in\R^{1+1}$. 
We highlight the role of the indicator functions in the integrand of \eqref{eq:DefRab}.
Without them, the resulting operator would have similar, but not identical, mapping properties to that of $\mathcal R_{\alpha,\beta}$, which by themselves do not appear sufficient to prove Theorem~\ref{thm:main}. 
Considering the adjoint operator,
\[
  \mathcal R_{\alpha,\beta}^*(F)(x)
  = (1+|x_1|)^{-\alpha} (1+|x_2|)^{-\beta} 
  \int_{\sph{1}} F(\omega)  \mathbbm{1}_{|\omega_1||x_1|\geq  1}\mathbbm{1}_{|\omega_2||x_2|\geq 1}
  e^{ix\cdot(|\omega_1|,|\omega_2|)} \d\sigma(\omega),
\] 
we investigate the $L^{q'}(\sph{1})$--$L^{p'}(\R^2)$ boundedness of 
$\mathcal R_{\alpha,\beta}^*$, 
and start with the important special case when $p'=p=2$; see also \cite{BS92}.  
 
\begin{proposition}\label{prop:GenSB}
Let $2\leq q<\infty$ and $\alpha,\beta>0$ be such that $\frac12<\alpha+\beta<1$.
Then $\mathcal R_{\alpha,\beta}^*:L^{q'}(\sph{1})\to L^2(\R^2)$ is  bounded if
$\alpha+\beta+\min\{\alpha,\beta\} \geq \frac{3}{2}-\frac{1}{q}$.
\end{proposition} 
\noindent The main tool is the oscillatory integral estimate from Lemma \ref{prop:OscillatoryIntegral} 
which together with elementary considerations place us in the setting of the weighted
Hardy--Littlewood--Sobolev inequality, Lemma \ref{lem:MappingS}.
\begin{proof}[Proof of Proposition \ref{prop:GenSB}] 
By symmetry, we may assume that $\alpha\geq\beta$, and
by interpolation, that $\alpha,\beta\neq\frac{1}{2}$.
We can further assume that $F\in L^{q'}(\sph{1})$ is supported in the region 
$\{\omega\in\sph{1}:\omega_1,\omega_2\geq 0\}$, 
since the other contributions can be estimated in a similar way.
For such $F$, set $F_\star(\phi):=F(\cos\phi,\sin\phi)$, and compute:
\begin{align*}
\|&\mathcal R_{\alpha,\beta}^*(F)\|_{L^2(\R^2)}^2
=\int_0^{\frac{\pi}{2}}\int_0^{\frac{\pi}{2}} F_\star(\varphi) \overline{F_\star(\phi)} 
\left(\int_{\min\{\cos(\phi),\cos(\varphi)\} |x_1|\geq 1} (1+|x_1|)^{-2\alpha}  
e^{ix_1(\cos\phi-\cos\varphi)} \d x_1\right) \\
&\times \left(\int_{\min\{\sin(\phi),\sin(\varphi)\}|x_2|\geq 1} (1+|x_2|)^{-2\beta}
e^{ix_2(\sin\phi-\sin\varphi)} \d x_2\right) \d\phi\d\varphi \\
&= 2\sum_{j=1}^3 \int_{(\varphi,\phi)\in I_j} |F_\star(\varphi)|  |F_\star(\phi)|
\left|\int_{|x_1|\geq \cos(\varphi)^{-1}} (1+|x_1|)^{-2\alpha}  
e^{ix_1(\cos\phi-\cos\varphi)} \d x_1\right| \\
&\times \left|\int_{|x_2|\geq \sin(\phi)^{-1}} (1+|x_2|)^{-2\beta}
e^{ix_2(\sin\phi-\sin\varphi)} \d x_2\right| \d\phi\d\varphi,  
\end{align*}
where the regions $\{I_j\}_{j=1}^3\subset [0,\frac{\pi}2]^2$ are defined as follows:
\begin{align*}
  I_1 &:= \left\{0\leq \varphi\leq \frac{\pi}{4},0\leq \phi\leq \varphi\right\},\\
  I_2 &:= \left\{\frac{\pi}{4}\leq \varphi\leq \frac{\pi}{2},0\leq \phi\leq \frac{\pi}{8}\right\},\\
  I_3 &:= \left\{\frac{\pi}{4}\leq \varphi\leq \frac{\pi}{2},\frac{\pi}{8}\leq \phi\leq \varphi\right\}.
\end{align*}

All the resulting oscillatory integrals will be estimated via Lemma~\ref{prop:OscillatoryIntegral}.
The ones over the region $I_2$ are easy to handle since 
$|\cos\phi-\cos\varphi|,|\sin\phi-\sin\varphi|\geq c$, for some $c>0$ and every $(\varphi,\phi)\in I_2$. 
Therefore the contribution from $I_2$ is bounded
by a constant multiple of $\|F\|_{L^1(\sph{1})}^2=O(\|F\|_{L^{q'}(\sph{1})}^2)$. 
\medskip

To estimate the integrals over the regions
$I_1,I_3$, we make use of the following lower bounds, valid for some $c>0$ and every $0\leq \phi\leq \varphi\leq \frac{\pi}{2}$:  
 \begin{align} \label{eq:CosSinEstimatesI}
   \begin{aligned}
  |\cos\varphi-\cos\phi|
  &\geq \cos\left(\frac{\varphi+\phi}{2}\right)-\cos\varphi 
  \;\geq\;\frac{\varphi-\phi}{2} \sin\left(\frac{\varphi+\phi}{2}\right) \\
  &\geq c \frac{\varphi-\phi}{2} \frac{\varphi+\phi}{2} 
  \;\geq\; \frac{c}{4} (\varphi-\phi) \varphi, \\
  |\sin\varphi-\sin\phi|
  &\geq \frac{c}{4}(\varphi-\phi)\left(\frac{\pi}{2}-\phi\right).  
\end{aligned}
 \end{align}
Similarly, for some $C<\infty$ and  every $0\leq \phi\leq \varphi\leq \frac{\pi}{2}$, we have that: 
\begin{align} \label{eq:CosSinEstimatesII}
  |\cos\varphi-\cos\phi|
  \leq C(\varphi-\phi)\varphi, \qquad 
  |\sin\varphi-\sin\phi| \leq C(\varphi-\phi)\left(\frac{\pi}{2}-\phi\right).  
\end{align}

\medskip

\noindent 
\textit{Analysis on $I_1$:}
If $0<\alpha,\beta<\frac{1}{2}$, then
Lemma~\ref{prop:OscillatoryIntegral} and the bounds \eqref{eq:CosSinEstimatesI} together imply:
\begin{align}
  &\int_{(\varphi,\phi)\in I_1} |F_\star(\varphi)|  |F_\star(\phi)|
 \left|\int_{|x_1|\geq \cos(\varphi)^{-1}} (1+|x_1|)^{-2\alpha}  
e^{ix_1(\cos\phi-\cos\varphi)} \d x_1\right| \notag\\
&\times \left|\int_{|x_2|\geq \sin(\phi)^{-1}} (1+|x_2|)^{-2\beta}
e^{ix_2(\sin\phi-\sin\varphi)} \d x_2\right| \d\phi\d\varphi   \notag \\
&\les \int_{(\phi,\varphi)\in I_1} |F_\star(\varphi)|  |F_\star(\phi)|  
|\cos\phi-\cos\varphi|^{2\alpha-1} |\sin\phi-\sin\varphi|^{2\beta-1}\d\phi\d\varphi \notag\\
&\les \int_{(\phi,\varphi)\in I_1} |F_\star(\varphi)|  |F_\star(\phi)|  
|\varphi(\varphi-\phi)|^{2\alpha-1}
\left|\left(\frac{\pi}{2}-\phi\right)(\varphi-\phi)\right|^{2\beta-1}\d\phi\d\varphi\notag\\
&\les \int_0^{\frac{\pi}{4}}  |F_\star(\varphi)| \left( |\varphi|^{2\alpha-1} \int_0^\varphi |F_\star(\phi)|  
|\varphi-\phi|^{2\alpha+2\beta-2} \d\phi \right) \d\varphi,\label{eq:LastIntegral}
\end{align}
where we have used that $(\frac{\pi}2-\phi)^{2\beta-1}\lesssim1$ since $0\leq\phi\leq\frac{\pi}4$.
Setting $(a,b):=(1-2\alpha,2-2\alpha-2\beta)$, we have  $a\geq 0, 0<b<1$ and $a+b\leq\frac2{q}$. 
Therefore definition~\eqref{eq:defOpS}, H\"older's inequality, and Lemma~\ref{lem:MappingS} 
(which can be applied since $q\geq 2$ and thus $q\geq q'$)
together imply that \eqref{eq:LastIntegral} is bounded by 
\[
\int_0^{\frac{\pi}4} |F_\star(\varphi)| \mathcal S_{a,b}(|F_\star|)(\varphi)\d\varphi
\leq\|F_\star\|_{L^{q'}([0,\frac{\pi}4])}\|\mathcal S_{a,b}(|F_\star|)\|_{L^q([0,\frac{\pi}4])} 
\lesssim \|F_\star\|_{L^{q'}([0,\frac{\pi}4])}^2 \leq \|F\|_{L^{q'}(\sph{1})}^2.
\]
If $0<\beta<\frac{1}{2}<\alpha$, then
Lemma~\ref{prop:OscillatoryIntegral} and estimates \eqref{eq:CosSinEstimatesI}--\eqref{eq:CosSinEstimatesII} together imply:
\begin{align}
  &\int_{(\varphi,\phi)\in I_1} |F_\star(\varphi)|  |F_\star(\phi)|
 \left|\int_{|x_1|\geq \cos(\varphi)^{-1}} (1+|x_1|)^{-2\alpha}  
e^{ix_1(\cos\phi-\cos\varphi)} \d x_1\right| \notag\\
&\times \left|\int_{|x_2|\geq \sin(\phi)^{-1}} (1+|x_2|)^{-2\beta}
e^{ix_2(\sin\phi-\sin\varphi)} \d x_2\right| \d\phi\d\varphi    \notag\\
&\les \int_{(\phi,\varphi)\in I_1} |F_\star(\varphi)|  |F_\star(\phi)|  
|\cos\phi|^{2\alpha-1} |\sin\phi-\sin\varphi|^{2\beta-1}\d\phi\d\varphi \notag\\
&\les \int_{(\phi,\varphi)\in I_1} |F_\star(\varphi)|  |F_\star(\phi)|  
\left(\frac{\pi}{2}-\phi\right)^{2\alpha-1} \left|\left(\frac{\pi}{2}-\phi\right)(\varphi-\phi)\right|^{2\beta-1}\d\phi\d\varphi
\notag\\
&\les \int_0^{\frac{\pi}{4}}  |F_\star(\varphi)| \left(\int_0^\phi |F_\star(\varphi)|  
|\varphi-\phi|^{2\beta-1} \d\phi \right)\d\varphi.\label{eq:LastInt2} 
\end{align}
Setting $(a,b):=(0,1-2\beta)$, we have $a\geq 0, 0<b<1$, and $a+b\leq\frac2{q}$; indeed, setting $\gamma:=\alpha+\beta+\min\{\alpha,\beta\}$,
 it follows that
\[  \frac{2}{q}-a-b
  \geq 3-2\gamma - (1-2\beta)
  = 2 - 2(\alpha+\beta) \geq 0.\]
Lemma~\ref{lem:MappingS} again implies that the integral \eqref{eq:LastInt2} is $O(\|F\|_{L^{q'}(\sph{1})}^2)$.

\medskip

\noindent \textit{Analysis on $I_3$:}
If $0<\alpha,\beta<\frac{1}{2}$, then Lemma~\ref{prop:OscillatoryIntegral}, the bounds \eqref{eq:CosSinEstimatesI}, and the change of variables
$(\frac{\pi}2-\varphi,\phi)\rightsquigarrow(\varphi,\frac{\pi}2-\phi)$ together yield
\begin{align*} 
  &\int_{(\varphi,\phi)\in I_3} |F_\star(\varphi)|  |F_\star(\phi)| 
  \left|\int_{|x_1|\geq \cos(\varphi)^{-1}} (1+|x_1|)^{-2\alpha}  
e^{ix_1(\cos\phi-\cos\varphi)} \d x_1\right| \\
&\times \left|\int_{|x_2|\geq \sin(\phi)^{-1}} (1+|x_2|)^{-2\beta}
e^{ix_2(\sin\phi-\sin\varphi)} \d x_2\right| \d\phi\d\varphi    \\
 &\les  \int_{\frac{\pi}4}^{\frac{\pi}2}  |F_\star(\varphi)| 
\left(
\int_{\frac{\pi}8}^{\varphi}(\varphi-\phi)^{2\alpha+2\beta-2}
\left(\frac{\pi}2-\phi\right)^{2\beta-1}
|F_\star(\phi)|\d\phi\right)\d\varphi \\ 
&=\int_0^{\frac{\pi}4}  \left|F_\star\left(\frac{\pi}2-\varphi\right)\right| \left(
\int_{\varphi}^{\frac{\pi}2}(\varphi-\phi)^{2\alpha+2\beta-2}
\phi^{2\beta-1}\left|F_\star\left(\frac{\pi}2-\phi\right)\right|
\d\phi\right)\d\varphi\\
&=\int_0^{\frac{\pi}4} |\tilde F_\star(\varphi)| \mathcal S^*_{a,b}(|\tilde F_\star|)(\varphi)\d\varphi
\lesssim \|\tilde F_\star\|_{L^{q'}([0,\frac{\pi}4])}^2 \leq \|F\|_{L^{q'}(\sph{1})}^2,
\end{align*}
where $\tilde F_\star:=F_\star(\frac{\pi}2-\cdot)$. Here we used the $L^{q'}$--$L^q$
bound for the adjoint operator $\mathcal S_{a,b}^*$, recall \eqref{eq:AdjointS}, implied by Lemma~\ref{lem:MappingS} with
$(a,b):=(1-2\beta,2-2\alpha-2\beta)$. 
The analysis of the case $\beta<\frac{1}{2}<\alpha$ proceeds along similar lines, and is therefore
omitted.
This concludes the proof of the proposition.
\end{proof}
  
Next we extend the range of boundedness of $\mathcal R_{\alpha,\beta}^*$ given by Proposition \ref{prop:GenSB} via
interpolation with a trivial estimate for $\mathcal R_{0,0}^*$.

\begin{proposition}\label{prop:SteinInterpolation}
  Let ${1< p}\leq 2\leq q<\infty$ and  $\alpha,\beta>0$ be such that
  $\frac{1}{p'}<\alpha+\beta<\frac2{p'}$.
  Then $\mathcal R_{\alpha,\beta}^*:L^{q'}(\sph{1})\to L^{p'}(\R^2)$ is bounded if 
  $\alpha+\beta+\min\{\alpha,\beta\} \geq \frac{3}{p'}-\frac{1}{q}$. 
\end{proposition}
\begin{proof}
Set $\gamma:=\alpha+\beta+\min\{\alpha,\beta\}$. {We use complex interpolation for} the analytic family of
operators given by
  \begin{equation}\label{eq:analfamop}
  \mathcal E_s := \mathcal R_{\frac{p'\alpha s}2, \frac{p'\beta s}2}^*,
  \end{equation} where $s\in S:=\{z\in\Co: 0\leq\Re(z)\leq 1\}$. 
 Start by noting that,  given simple functions $F\in L^1(\sph{1})$ and $g\in L^1(\R^2)$, the map
 \begin{equation}\label{eq:SteinMap}
 s\mapsto \int_{\R^2} \mathcal E_s(F)(x) g(x) \d x
 \end{equation} 
 is analytic in the interior of the strip $S$, continuous on $S$, and moreover the function defined by \eqref{eq:SteinMap} is uniformly bounded
 above in $S$.
 \medskip
 
 Now, Proposition~\ref{prop:GenSB}, with $(\alpha,\beta)$ replaced by $(p'\alpha/2,p'\beta/2)$, yields
   \begin{equation}\label{eq:IntEndPt1}
   \|\mathcal E_s(F)\|_{L^2(\R^2)} \lesssim\|F\|_{L^{q_1'}(\sph{1})}
   \text{ for } \Re(s)=1, \;\text{if } \frac12\leq \frac{1}{q_1'}\leq \frac{p'}{2}\gamma-\frac{1}{2}.
   \end{equation}
 On the other hand, we have the following trivial estimate:
      \begin{equation}\label{eq:IntEndPt2}
      \|\mathcal E_s(F)\|_{L^\infty(\R^2)} \lesssim\|F\|_{L^{q_0'}(\sph{1})}
   \text{ for } \Re(s)=0, \;\text{if } 0\leq \frac{1}{q_0'}\leq 1.
   \end{equation}
 Since $q\geq p$ and $\gamma\geq \frac{3}{p'}-\frac{1}{q}$, we may choose $q_0,q_1$ satisfying the above conditions, 
 with the additional property that $\frac{1}{q}=\frac{1-\theta}{q_0}+\frac{\theta}{q_1}$ for $\theta:=\frac{2}{p'}$. 
 Then Stein's interpolation theorem \cite[p.~205]{SW} can be applied to the analytic family of operators $\{\mathcal E_s\}_{s\in S}$ given by \eqref{eq:analfamop}, 
resulting from  \eqref{eq:IntEndPt1}--\eqref{eq:IntEndPt2} that
 \[
   \|\mathcal E_s(F)\|_{L^{p'}(\R^2)} \lesssim\|F\|_{L^{q'}(\sph{1})},
   \text{ for } \Re(s)=\theta=\frac{2}{p'}.
 \]
 This amounts to the desired conclusion since $\theta=\frac{2}{p'}$ implies $\mathcal E_\theta= \mathcal R_{\alpha,\beta}^*$. 
\end{proof}

Proposition \ref{prop:SteinInterpolation} will later on be refined via interpolation with a non-trivial estimate for the operator $\mathcal R_{0,0}^*$.
Since $\mathcal R_{0,0}^*$ is similar but not identical to the two-dimensional extension operator on
the unit circle, we first need to prove the latter estimate.  
That is the content of our next result.

\begin{proposition}\label{prop:ZygmundRange}
Let $1\leq p<\frac43$ and $1\leq q\leq \frac{p'}3$.
Then $\mathcal R_{0,0}^*:L^{q'}(\sph{1})\to L^{p'}(\R^2)$ defines a bounded operator.
\end{proposition}
\begin{proof}
We may assume $p>1$, and will bound
 $\mathcal R_{0,0}:L^p(\R^2)\to L^q(\sph{1})$ instead. Given $f:\R^2\to\Co$, set 
$g_{x_1}(x_2)  := g(x_1,x_2)  := f(x_1,x_2) {\mathbbm{1}_{x_1^{-2}+x_2^{-2}\leq 1}}.$
Then
\begin{align*}
  \mathcal R_{0,0}(f)(\omega)
  &= \int_{\frac{1}{|\omega_1|}}^\infty  \int_{\frac{1}{|\omega_2|}}^\infty f(x_1,x_2)
  e^{-i(x_1|\omega_1|+x_2|\omega_2|)} \d x_1\d x_2  \\
  &= \int_{\frac{1}{|\omega_1|}}^\infty  \int_{\R} g(x_1,x_2) e^{-i(x_1|\omega_1|+x_2|\omega_2|)}  \d x_1\d x_2 \\
  &= \int_{\R^2} g(x) e^{-ix\cdot(|\omega_1|,|\omega_2|)}  \d x
    - \int_1^{\frac{1}{|\omega_1|}} \widehat{g_{x_1}}(|\omega_2|) e^{-ix_1|\omega_1|}\d x_1.
\end{align*}
Within the desired range of exponents, bounds for the first term  are well-known \cite{CS72, Zy74}, so we proceed to bound the second term in $L^q(\sph{1})$:
\begin{align}
\Big( \int_{\sph{1}} \left(\int_1^{\frac1{|\omega_1|}}
|\widehat{g_{x_1}}(|\omega_2|)| \d x_1\right)^q&\d\sigma(\omega)\Big)^{\frac{1}{q}} 
= \Big(\int_0^1 (1-r^2)^{-\frac12} \left(\int_1^{\frac1{\sqrt{1-r^2}}}
|\widehat{g_{x_1}}(r)| \d x_1\right)^q\d r \Big)^{\frac{1}{q}} \notag\\
&= \sup_{\|h\|_{L^{q'}}=1} 
\int_0^1 (1-r^2)^{-\frac1{2q}}  \left(\int_1^{\frac1{\sqrt{1-r^2}}}
|\widehat{g_{x_1}}(r)| \d x_1 \right) h(r) \d r \notag\\
&= \sup_{\|h\|_{L^{q'}}=1}  
 \int_1^\infty 
\left(\int_{\sqrt{1-x_1^{-2}}}^1 (1-r^2)^{-\frac1{2q}} |\widehat{g_{x_1}}(r)| h(r) \d r \right) \d x_1 \notag\\ 
&\leq   
 \int_1^\infty 
\left(\int_{\sqrt{1-x_1^{-2}}}^1 (1-r^2)^{-\frac{p'}{2(p'-q)}}  \d r \right)^{\frac{1}{q}-\frac{1}{p'}} 
  \|\widehat{g_{x_1}}\|_{L^{p'}}  \d x_1. \label{leq:LastTerm} 
  \end{align}
  Here we changed variables, used duality, Fubini's Theorem, and H\"older's inequality.
Another change of variables, the Hausdorff--Young inequality, and H\"older's inequality in Lorentz space \cite[Theorem 3.4]{ON63} together yield the following upper bound for \eqref{leq:LastTerm}:
\begin{align*}
 \int_1^\infty 
&\left(\int_0^{1-\sqrt{1-x_1^{-2}}} s^{-\frac{p'}{2(p'-q)}}  \d s \right)^{\frac{1}{q}-\frac{1}{p'}} 
  \|g_{x_1}\|_{L^p}\d x_1 \\
&\les  
 \int_1^\infty \left[\left(1-\sqrt{1-x_1^{-2}}\right)^{1-\frac{p'}{2(p'-q)}} \right]^{\frac{1}{q}-\frac{1}{p'}} 
  \|g_{x_1}\|_{L^p}\d x_1 \\
&\les  
 \int_1^\infty (x_1^{-2})^{\frac{1}{q}-\frac{1}{p'}-\frac{1}{2q}}   \|g_{x_1}\|_{L^p}\d x_1  \\
&\leq \| (\cdot)^{-\frac{1}{q}+\frac{2}{p'}} \|_{L^{p',\infty}([1,\infty))}
\|g\|_{L^{p,1}(\R^2)}. 
\end{align*}
The first term on the right-hand side is finite  since $q\leq\frac{p'}3$. 
Since $|g|\leq|f|$,
this establishes the $L^{p,1}(\R^2)$--$L^q(\sph{1})$ boundedness of $\mathcal R_{0,0}$, provided $1\leq p<\frac{4}{3}$ and $1\leq q\leq \frac{p'}{3}$. Real interpolation \cite[Theorem 5.3.2]{BL} within this family of Lorentz space estimates
 and compactness of $\sph{1}$ together yield the claimed strong type estimates.
This concludes the proof of the proposition.
\end{proof}


\section{Proof of Theorem \ref{thm:main}}\label{sec:ProofThm1}

In this section, we prove Theorem \ref{thm:main}.
After some preliminary simplifications, we reduce the analysis to three main estimates, which we address separately.

\medskip

 Fix $d\geq 4$ and $k\in\{2,3,\ldots, d-2\}$, and set $m:=\min\{k,d-k\}$.
By a routine density argument, it suffices to establish the estimate 
\begin{equation}\label{eq:GoalThm1}
\|\widehat{f}\|_{L^2(\sph{d-1})}\leq C(k,d,p) \|f\|_{L^{p}(\R^d)},
\qquad 1\leq p\leq \frac{2(d+m)}{d+m+2},
\end{equation}
for every $G_k$-symmetric Schwartz function $f:\R^d\to\Co$.
Our starting point is the following formula from Lemma \ref{lem:GkSymFT}:
\begin{equation}\label{eq:GkFT}
\widehat{f}(\eta,\zeta)=(2\pi)^{\frac d2} |\eta|^{\frac{2-d+k}{2}} |\zeta|^{\frac{2-k}2} \int_0^\infty \int_0^\infty \rho_1^{\frac{d-k}2}\rho_2^{\frac k2}  f_0(\rho_1,\rho_2) J_{\frac{d-k-2}{2}}(\rho_1|\eta|)J_{\frac{k-2}2}(\rho_2|\zeta|)\d\rho_1\d\rho_2,
\end{equation}
which holds  at every $(\eta,\zeta)\in\R^{d-k}\times\R^k$, for any $G_k$-symmetric Schwartz function $f:\R^d\to\Co$.
In light of Lemma \ref{lem:HardBessel}, there exist nonzero constants $A_1,A_2\in\Co\setminus\{0\}$ and
functions $R_1,R_2:(0,\infty)\to\Co$, such that, for every $r\geq 0$,
\begin{align}
J_{\frac{d-k-2}{2}}(r)&=(A_1 e^{ir}+\overline{A_1} e^{-ir})r^{-\frac12}\mathbbm{1}_{[1,\infty)}(r)+R_1(r),\label{eq:Dec1}\\
J_{\frac{k-2}2}(r)&=(A_2 e^{ir}+\overline{A_2} e^{-ir})r^{-\frac12}\mathbbm{1}_{[1,\infty)}(r)+R_2(r).\label{eq:Dec2}
\end{align}
Moreover, the following estimates hold, for every $r\geq 0$:
\begin{align}
|R_1(r)|&\lesssim r^{\frac{d-k-2}{2}}(1+r)^{\frac{k-d-1}{2}},\label{eq:EstR1}\\
|R_2(r)|&\lesssim  r^{\frac{k-2}2}(1+r)^{-\frac{k+1}2}.\label{eq:EstR2}
\end{align}
The decomposition \eqref{eq:Dec1}--\eqref{eq:Dec2} induces a splitting in \eqref{eq:GkFT},
\begin{equation}\label{eq:splitting}
\widehat{f}=(2\pi)^{\frac d2}\left(\widehat{f}_1+\sum_{j=2}^5 (\widehat{f}_j+\overline{\widehat{f}_j})\right),
\end{equation}
where each of the pieces is defined as follows:
\begin{align}
\widehat{f_1}(\eta,\zeta)&:=|\eta|^{\frac{k-d+2}{2}} |\zeta|^{\frac{2-k}2} \int_0^\infty \int_0^\infty \rho_1^{\frac{d-k}2}\rho_2^{\frac k2}  f_0(\rho_1,\rho_2)R_1(\rho_1|\eta|)R_2(\rho_2|\zeta|)\d\rho_2\d\rho_1;\label{eq:defF1}\\
\widehat{f_2}(\eta,\zeta)&:=A_2|\eta|^{\frac{k-d+2}{2}} |\zeta|^{\frac{1-k}2} \int_0^\infty \int_{\frac1{|\zeta|}}^\infty \rho_1^{\frac{d-k}2}\rho_2^{\frac {k-1}2}  f_0(\rho_1,\rho_2)R_1(\rho_1|\eta|)e^{i\rho_2|\zeta|}\d\rho_2\d\rho_1;\label{eq:defF2}\\
\widehat{f_3}(\eta,\zeta)&:=A_1|\eta|^{\frac{k-d+1}{2}} |\zeta|^{\frac{2-k}2} \int_{\frac1{|\eta|}}^\infty \int_0^\infty \rho_1^{\frac{d-k-1}2}\rho_2^{\frac k2}  f_0(\rho_1,\rho_2)e^{i\rho_1|\eta|}R_2(\rho_2|\zeta|)\d\rho_2\d\rho_1;\notag\\
\widehat{f_4}(\eta,\zeta)&:=A_1A_2|\eta|^{\frac{k-d+1}{2}} |\zeta|^{\frac{1-k}2} \int_{\frac1{|\eta|}}^\infty \int_{\frac1{|\zeta|}}^\infty \rho_1^{\frac{d-k-1}2}\rho_2^{\frac {k-1}2}  f_0(\rho_1,\rho_2)e^{i(\rho_1|\eta|+\rho_2|\zeta|)}\d\rho_2\d\rho_1;\label{eq:defF4}\\
\widehat{f_5}(\eta,\zeta)&:=A_1\overline{A_2}|\eta|^{\frac{k-d+1}{2}} |\zeta|^{\frac{1-k}2} \int_{\frac1{|\eta|}}^\infty \int_{\frac1{|\zeta|}}^\infty \rho_1^{\frac{d-k-1}2}\rho_2^{\frac {k-1}2}  f_0(\rho_1,\rho_2)e^{i(\rho_1|\eta|-\rho_2|\zeta|)}\d\rho_2\d\rho_1.\label{eq:defF5}
\end{align}
Note that each $\widehat{f_j}$ is $G_k$-symmetric.
We proceed to find suitable bounds for $\|\widehat{f_j}\|_{L^2(\sph{d-1})}$ for each $1\leq j\leq 5$.
By interchanging $d-k$ and $k$,
the estimates for $\widehat{f_3}, \widehat{f_5}$ are  analogous to those for $\widehat{f_2}, \widehat{f_4}$, respectively, 
and so the analysis actually reduces to three cases.

\medskip

Recall that, for a given $G_k$-symmetric function $f:\sph{d-1}\to\Co$, we set $f_0(|\eta|,|\zeta|)=f(\eta,\zeta)$, and have, for some dimensional constant $c_d\in (0,\infty)$ whose exact value will be unimportant,
\begin{equation}\label{eq:Norm}
  \|f\|_{L^p(\R^d)}^p = c_d^p \int_0^\infty\int_0^\infty \rho_1^{d-k-1}\rho_2^{k-1}
  |f_0(\rho_1,\rho_2)|^p\d \rho_1\d \rho_2. 
\end{equation}

\subsection{Estimating $\widehat{f}_1$}
This is by far the easiest case to handle.
\begin{proposition}\label{prop:f1}
For every $1\leq p<2$, there exists $C=C(k,d,p)<\infty$ such that 
\[\|\widehat{f_1}\|_{L^2(\sph{d-1})}\leq C\|f\|_{L^p(\R^d)},\]
for every $G_k$-symmetric Schwartz function $f:\R^d\to\Co$.
\end{proposition}
\begin{proof}
Fix $p\in[1,2)$. From definition \eqref{eq:defF1} of $\widehat{f_1}$ and estimates \eqref{eq:EstR1}--\eqref{eq:EstR2}, it follows that
\begin{align*}
|\widehat{f_1}(\eta,\zeta)|
&\lesssim \int_0^\infty \int_0^\infty \rho_1^{d-k-1}\rho_2^{k-1}  |f_0(\rho_1,\rho_2)|(1+\rho_1|\eta|)^{\frac{k-d-1}2}(1+\rho_2|\zeta|)^{-\frac{k+1}2}\d\rho_2\d\rho_1\\
&\leq\left(\int_0^\infty \int_0^\infty \rho_1^{d-k-1}\rho_2^{k-1} (1+\rho_1|\eta|)^{\frac{p'(k-d-1)}2}(1+\rho_2|\zeta|)^{-\frac{p'(k+1)}2}\d\rho_2\d\rho_1\right)^{\frac1{p'}} \|f\|_{L^p(\R^d)},
\end{align*}
where the last line follows from an application of H\"older's inequality and~\eqref{eq:Norm}.
The change of variables $(\rho_1|\eta|,\rho_2|\zeta|)\rightsquigarrow(\rho_1,\rho_2)$ then reveals that 
\begin{equation}\label{eq:LaterAlso}
|\widehat{f_1}(\eta,\zeta)|\lesssim |\eta|^{\frac{k-d}{p'}} |\zeta|^{-\frac{k}{p'}} \|f\|_{L^p(\R^d)},
\end{equation}
for every $(\eta,\zeta)\in\R^{d-k}\times\R^k$.
The latter estimate can be integrated over the unit sphere $|\eta|^2+|\zeta|^2=1$ via Lemma \ref{lem:SphInt}, yielding:
\begin{align*}
\int_{\sph{d-1}} |\widehat{f_1}(\eta,\zeta)|^2\d\sigma(\eta,\zeta)
&\lesssim \int_0^1 r^{d-k-1}(1-r^2)^{\frac{k-2}2} \left(r^{\frac{k-d}{p'}}(1-r^2)^{-\frac{k}{2p'}} \|f\|_{L^p(\R^d)}\right)^2 \d r\\
&=\left(\int_0^1 r^{2(d-k)(\frac12-\frac1{p'})-1} (1-r^2)^{k(\frac12-\frac1{p'})-1} \d
r\right)\|f\|_{L^p(\R^d)}^2\lesssim \|f\|_{L^p(\R^d)}^2
\end{align*}
since $p<2$. 
This concludes the proof of Proposition \ref{prop:f1}.
\end{proof}

\subsection{Estimating $\widehat{f}_2$ and $\widehat{f}_3$}
The estimates in this section will follow, in the spirit of Carleson--Sj\"olin \cite{CS72}, 
from a successive application of weighted versions of the Hausdorff--Young and the Hardy--Littlewood--Sobolev inequalities. 

\begin{proposition}\label{prop:f2}
There exists $C=C(k,d,p)<\infty$ such that 
\begin{align*}
  \|\widehat{f_2}\|_{L^2(\sph{d-1})}
  &\leq C\|f\|_{L^{p}(\R^d)},\qquad\text{if }1\leq p\leq \frac{2(d+k)}{d+k+2}, \\
  \|\widehat{f_3}\|_{L^2(\sph{d-1})}
  &\leq C\|f\|_{L^{p}(\R^d)},\qquad\text{if  }1\leq p\leq \frac{2(2d-k)}{2d-k+2}, 
\end{align*}
for every $G_k$-symmetric Schwartz function $f:\R^d\to\Co$.
\end{proposition}
\begin{proof}
We focus on the estimate for $f_2$ because the analysis of $f_3$ is analogous up
to interchanging the roles of $k, d-k$. 
By interpolation with the trivial estimate at $p=1$, it suffices to consider the endpoint case $p=p_\star:=\frac{2(d+k)}{d+k+2}$. 
 Set $\delta:=(k-1)(\frac{1}{p_\star}-\frac{1}2)$, and define 
\begin{equation}\label{eq:defG}
g(\rho_1,\rho_2):=\rho_1^{\frac{d-k-1}{p_\star}}\rho_2^{\frac{k-1}{p_\star}} f_0(\rho_1,\rho_2),
\end{equation} 
for each $\rho_1,\rho_2>0$, as well as the corresponding ``slice'' function $g_{\rho_1}:(0,\infty)\to\Co$,
\begin{equation}\label{eq:defG1}
g_{\rho_1}(\rho_2):=\rho_2^{-\delta}g(\rho_1,\rho_2)\mathbbm{1}_{[1,\infty)}(\rho_2).
\end{equation} 
From the definitions \eqref{eq:defF2} of $\widehat{f_2}$ and \eqref{eq:defG}--\eqref{eq:defG1} of
$g, g_{\rho_1}$, respectively, it follows that
\begin{align*}
|\widehat{f_2}&(\eta,\zeta)|
\lesssim 
|\eta|^{\frac{2-d+k}{2}} |\zeta|^{\frac{1-k}2}\left| \int_0^\infty  \rho_1^{\frac{d-k}2-\frac{d-k-1}{p_\star}} R_1(\rho_1|\eta|) 
\left(\int_{\frac1{|\zeta|}}^\infty \rho_2^{-\delta}  g(\rho_1,\rho_2) e^{i\rho_2|\zeta|}\d\rho_2\right)\d\rho_1\right|\\
&\leq|\eta|^{\frac{2-d+k}{2}} |\zeta|^{\frac{1-k}2} \int_0^\infty  \rho_1^{\frac{d-k}2-\frac{d-k-1}{p_\star}}
|R_1(\rho_1|\eta|)| \left(|\widehat{g}_{\rho_1}(|\zeta|)|+\int_1^{\frac1{|\zeta|}} \rho_2^{-\delta}  |g(\rho_1,\rho_2)|\d\rho_2\right)\d\rho_1,
\end{align*}
for every $(\eta,\zeta)\in\R^{d-k}\times\R^k$.
Estimate \eqref{eq:EstR1} then implies
\begin{align*}
|\widehat{f_2}(\eta,\zeta)|
&\lesssim
 |\zeta|^{\frac{1-k}2} \int_0^\infty  \rho_1^{\frac{d-k-1}{p_\star'}} (1+\rho_1|\eta|)^{\frac{k-d-1}2} 
\left(|\widehat{g}_{\rho_1}(|\zeta|)|+\int_1^{\frac1{|\zeta|}} \rho_2^{-\delta}  |g(\rho_1,\rho_2)|\d\rho_2\right)\d\rho_1\\
&\lesssim
 |\zeta|^{\frac{1-k}2} \int_0^\infty  \rho_1^{\frac{d-k-1}{p_\star'}} (1+\rho_1|\eta|)^{\frac{k-d-1}2} 
\left(|\widehat{g}_{\rho_1}(|\zeta|)|+{|\zeta|}^{\delta-\frac1{p_\star'}}  \|g(\rho_1,\cdot)\|_{L^{p_\star}_{\rho_2}(\R_+)}\right)\d\rho_1,
\end{align*}
where the last line follows from an application of H\"older's inequality.
We break up the latter integral into two pieces, and analyze both summands $\textup{I}, \textup{II}$ (defined in \eqref{eq:DefI}, \eqref{eq:DefII} below) separately.
The second one is easier to handle, and can be bounded using H\"older's inequality as follows:
\begin{align}\label{eq:DefII}
  \begin{aligned}
\textup{II}(\eta,\zeta)
&:= |\zeta|^{\frac{1-k}2+\delta-\frac1{p_\star'}} \int_0^\infty 
 \rho_1^{\frac{d-k-1}{p_\star'}} (1+\rho_1|\eta|)^{\frac{k-d-1}2}
\|g(\rho_1,\cdot)\|_{L^{p_\star}_{\rho_2}(\R_+)}\d\rho_1\\
  &\lesssim
  |\eta|^{\frac{k-d}{p_\star'}}  |\zeta|^{-\frac{k}{p_\star'}}  \|g\|_{L^{p_\star}(\R_+^2)}
  \simeq|\eta|^{\frac{k-d}{p_\star'}}|\zeta|^{-\frac{k}{p_\star'}} \|f\|_{L^{p_\star}(\R^d)}.
\end{aligned}
\end{align}
Comparing with \eqref{eq:LaterAlso}, we see that this piece can be handled exactly as in the proof of
Proposition~\ref{prop:f1}, resulting in the following bound:
 \[\int_{\sph{d-1}}\textup{II}^2(\eta,\zeta)\d\sigma(\eta,\zeta)\lesssim \|f\|_{L^{p_\star}(\R^d)}^2.\]
 Here we only used $p_\star<2$.
It remains to estimate the integral of (the square of) the first summand, given by
\begin{equation}\label{eq:DefI} 
\textup{I}(\eta,\zeta):=|\zeta|^{\frac{1-k}2} \int_0^\infty  \rho_1^{\frac{d-k-1}{p_\star'}} (1+\rho_1|\eta|)^{\frac{k-d-1}2} 
|\widehat{g}_{\rho_1}(|\zeta|)|\d\rho_1.
\end{equation}
With that purpose in mind, specialize Lemma~\ref{lem:weightedHY} to $p=p_\star$
and 
\begin{equation*}
\frac1q=\frac1{q_\star}:=1-\frac1{p_\star}-\delta=\frac{k+1}2-\frac{k}{p_\star}=\frac{d-k}{2(d+k)},
\end{equation*}
to obtain the following estimate:
\begin{equation}\label{eq:AfterWHY}
\|\widehat{g}_{\rho_1}\|_{L^{q_\star}(\R_+)}
\lesssim \|g(\rho_1,\cdot)\mathbbm{1}_{[1,\infty)}\|_{L^{p_\star}(\R_+)}
\leq \|g(\rho_1,\cdot)\|_{L^{p_\star}(\R_+)}.
\end{equation}
From Lemma \ref{lem:SphInt}, we have that
\begin{align}
\int_{\sph{d-1}}&\textup{I}^2(\eta,\zeta)\d\sigma(\eta,\zeta)\notag\\
&\lesssim\int_0^1 r^{d-k-1}(1-r^2)^{-\frac12} 
\left( \int_0^\infty  \rho_1^{\frac{d-k-1}{p_\star'}} (1+\rho_1r)^{\frac{k-d-1}2} 
|\widehat{g}_{\rho_1}(\sqrt{1-r^2})|\d\rho_1\right)^2 \d r\notag\\
&=\int_0^1 (1-s^2)^{\frac{d-k-2}2} 
\left( \int_0^\infty  \rho_1^{\frac{d-k-1}{p_\star'}} (1+\rho_1\sqrt{1-s^2})^{\frac{k-d-1}2} 
|\widehat{g}_{\rho_1}(s)|\d\rho_1\right)^2 \d s\notag\\
&= \int_0^\infty \int_0^\infty  (\rho_1\tilde\rho_1)^{\frac{d-k-1}{p_\star'}}  \left(\int_0^1 K_{\rho_1,\tilde\rho_1}(s)|\widehat{g}_{\rho_1}(s)| 
|\widehat{g}_{\tilde\rho_1}(s)|\d s\right)\d\tilde\rho_1\d\rho_1 \notag\\
&= 2\int_0^\infty \int_0^{\rho_1}  (\rho_1\tilde\rho_1)^{\frac{d-k-1}{p_\star'}}  \left(\int_0^1 K_{\rho_1,\tilde\rho_1}(s)|\widehat{g}_{\rho_1}(s)| 
|\widehat{g}_{\tilde\rho_1}(s)|\d s\right)\d\tilde\rho_1\d\rho_1, \label{eq:TBC}
\end{align}
where, for each $(\rho_1,\tilde \rho_1)\in\R_+^2$, the function $K_{\rho_1,\tilde\rho_1}:[0,1]\to\R$ is defined via
\[K_{\rho_1,\tilde\rho_1}(s):=(1-s^2)^{\frac{d-k-2}2}(1+\rho_1\sqrt{1-s^2})^{\frac{k-d-1}2}(1+\tilde\rho_1\sqrt{1-s^2})^{\frac{k-d-1}2}.\]
Set $I:=[0,1]$,
and estimate the $L^{(\frac {q_\star}2)'}(I)$-norm of $K_{\rho_1,\tilde\rho_1}$ as follows: 
\begin{align}
\|K_{\rho_1,\tilde\rho_1}\|_{L^{(\frac {q_\star}2)'}(I)}&
= \left(\int_0^1
r^{\frac{q_\star(d-k-2)}{q_\star-2}+1}(1-r^2)^{-\frac12}(1+\rho_1r)^{\frac{q_\star(k-d-1)}{2(q_\star-2)}}(1+\tilde\rho_1r)^{\frac{q_\star(k-d-1)}{2(q_\star-2)}}
\d r\right)^{\frac{q_\star-2}{q_\star}}\notag\\
\lesssim& \left(\int_0^{\frac12}
r^{\frac{q_\star(d-k-2)}{q_\star-2}+1}(1+\rho_1r)^{\frac{q_\star(k-d-1)}{2(q_\star-2)}}(1+\tilde\rho_1r)^{\frac{q_\star(k-d-1)}{2(q_\star-2)}}
\d r\right)^{\frac{q_\star-2}{q_\star}} \notag\\
&+(1+\rho_1)^{\frac{k-d-1}{2}}(1+\tilde\rho_1)^{\frac{k-d-1}{2}}\notag\\
\leq& \left(\tilde\rho_1^{\frac{q_\star(k-d+2)}{q_\star-2}-2}\int_0^{\infty} t^{\frac{q_\star(d-k-2)}{q_\star-2}+1}
 \left(1+\frac{\rho_1}{\tilde\rho_1}t\right)^{\frac{q_\star(k-d-1)}{2(q_\star-2)}}(1+t)^{\frac{q_\star(k-d-1)}{2(q_\star-2)}}
 \d t\right)^{\frac{q_\star-2}{q_\star}}\notag\\
&+(1+\rho_1)^{\frac{k-d-1}{2}}(1+\tilde\rho_1)^{\frac{k-d-1}{2}}\notag\\
\lesssim&\rho_1^{\frac{k-d-1}2}\tilde\rho_1^{\frac{k-d+1}2+\frac4{q_\star}}+(1+\rho_1)^{\frac{k-d-1}{2}}(1+\tilde\rho_1)^{\frac{k-d-1}{2}},\label{eq:TwoSummands}
\end{align}
where in the last line we used the trivial bound $1+\frac{\rho_1}{\tilde\rho_1}t\geq \frac{\rho_1}{\tilde\rho_1}t$.
The contribution from the second summand in \eqref{eq:TwoSummands} is straightforward to handle, and so we focus on the first one.
Going back to \eqref{eq:TBC}, we then have that
\begin{align*}
\int_{\sph{d-1}}&\textup{I}^2(\eta,\zeta)\d\sigma(\eta,\zeta) \\
&\lesssim 
 \int_0^\infty \int_0^{\rho_1}  (\rho_1\tilde\rho_1)^{\frac{d-k-1}{p_\star'}}  
 \|K_{\rho_1,\tilde\rho_1}\|_{L^{(\frac {q_\star}2)'}(I)} 
 \|\widehat{g}_{\rho_1}\|_{L^{q_\star}(I)} \|\widehat{g}_{\tilde\rho_1}\|_{L^{q_\star}(I)}
 \d\tilde\rho_1\d\rho_1\\
&\lesssim 
 \int_0^\infty \int_0^{\rho_1}  (\rho_1\tilde\rho_1)^{\frac{d-k-1}{p_\star'}}  
\rho_1^{\frac{k-d-1}2}\tilde\rho_1^{\frac{k-d+1}2+\frac4{q_\star}} \|g(\rho_1,\cdot)\|_{L^{p_\star}(\R_+)} \|g(\tilde\rho_1,\cdot)\|_{L^{p_\star}(\R_+)} 
 \d\tilde\rho_1\d\rho_1\\
&= \int_0^\infty \|g(\rho_1,\cdot)\|_{L^{p_\star}(\R_+)} \left(\rho_1^{-a}\int_0^{\rho_1}  \tilde\rho_1^{-b}  
  \|g(\tilde\rho_1,\cdot)\|_{L^{p_\star}(\R_+)} \d\tilde\rho_1\right)\d\rho_1,
 \end{align*}
 where the second estimate  follows from \eqref{eq:AfterWHY} and \eqref{eq:TwoSummands}, and we set\footnote{Alternatively, $a=\frac{2d-1}{d+k}$ and $b=\frac{k-d-1}{d+k}$.}
 \[
   a:= \frac{d-k+1}2-\frac{d-k-1}{p_\star'},\qquad
   b:= (d-k-1)\left(\frac12-\frac1{p_\star'}\right)-\frac4{q_\star}.
 \]
 It is straightforward to check that these exponents satisfy 
 $a+b=\frac2{p_\star'}$ and 
  $b p_\star'<1$. 
  Since we also have that $1<p_\star\leq p_\star'<\infty$, Lemma~\ref{lem:MappingT} yields
 \begin{align*}
\int_{\sph{d-1}}\textup{I}^2(\eta,\zeta)\d\sigma(\eta,\zeta)
&\lesssim
\left\|\|g(\rho_1,\cdot)\|_{L^{p_\star}(\R_+)}\right\|_{L_{\rho_1}^{p_\star}(\R_+)}
\left\|\rho_1^{-a}\int_0^{\rho_1}  \tilde\rho_1^{-b}  
  \|g(\tilde\rho_1,\cdot)\|_{L^{p_\star}(\R_+)} \d\tilde\rho_1\right\|_{L_{\rho_1}^{p_\star'}(\R_+)}\\
&\lesssim
\left\|\|g(\rho_1,\cdot)\|_{L^{p_\star}(\R_+)}\right\|_{L_{\rho_1}^{p_\star}(\R_+)}^2 
=\|g\|_{L^{p_\star}(\R_+^2)}^2\simeq\|f\|_{L^{p_\star}(\R^d)}^2.
 \end{align*}
This concludes the proof of Proposition \ref{prop:f2}.
\end{proof}

\subsection{Estimating $\widehat{f}_4$ and $\widehat{f}_5$}
 The estimates in this section are the most delicate ones, but most of the work has been done already.
 We heavily rely on the weighted estimates for the restriction-type operator $\mathcal R_{\alpha,\beta}$, recall \eqref{eq:DefRab}, which we proved in  \S~\ref{sec:WRtypeE}, and record the following consequence.

\begin{corollary} \label{cor:WeightedCor2}
Let  $d\geq 4, k\in\{2,3,\ldots,d-2\}, m:=\min\{d-k,k\}$, and $1\leq  p\leq \frac{2(d+m)}{d+m+2}$.
Set 
\begin{equation}\label{eq:DefAlBe}
  \alpha_p := (d-k-1)\left(\frac{1}{p}-\frac{1}{2}\right),\quad
  \beta_p := (k-1)\left(\frac{1}{p}-\frac{1}{2}\right).
\end{equation}
Then $\mathcal R_{\alpha_p,\beta_p}:L^p(\R^2)\to L^2(\sph{1})$ defines a bounded operator.
\end{corollary}              
\begin{proof}
By interpolation and compactness of $\sph{1}$, it suffices to consider the endpoint case $p=\frac{2(d+m)}{d+m+2}$.
Proposition \ref{prop:SteinInterpolation} directly implies the desired conclusion in all situations of interest, except when $2k=4=d$, 
so we focus on that case.
Proposition \ref{prop:GenSB} implies 
\begin{equation}\label{eq:FirstInt}
\|\mathcal{R}_{\frac13,\frac13}^\ast(F)\|_{L^2(\R^2)}\lesssim \|F\|_{L^2(\sph{1})},
\end{equation}
and from Proposition \ref{prop:ZygmundRange} we have that
\begin{equation}\label{eq:SecondInt}
\|\mathcal{R}_{0,0}^\ast(F)\|_{L^6(\R^2)}\lesssim \|F\|_{L^2(\sph{1})}.
\end{equation}
 As in the proof of Proposition \ref{prop:SteinInterpolation}, Stein's interpolation theorem \cite[p.~205]{SW} can be applied to the analytic family of operators $\{\mathcal R_{\frac{s}3,\frac{s}3}^\ast: 0\leq\Re(s)\leq 1\}$, yielding from \eqref{eq:FirstInt}--\eqref{eq:SecondInt} 
  \[\|\mathcal{R}_{\frac16,\frac16}^\ast(F)\|_{L^3(\R^2)}\lesssim \|F\|_{L^2(\sph{1})}.\]
  This is equivalent, by duality, to the desired conclusion.
\end{proof}

The relevance of Corollary \ref{cor:WeightedCor2} becomes apparent once we note that from definitions \eqref{eq:defF4}, \eqref{eq:defF5} it follows that:
 \begin{align}
   \widehat{f_4}(\eta,\zeta)
   &= A_1A_2|\eta|^{\frac{k-d+1}{2}} |\zeta|^{\frac{1-k}2}
   \mathcal R_{\alpha_p,\beta_p}(h)(-|\eta|,-|\zeta|),\label{eq:f4IntermsR}\\
   \widehat{f_5}(\eta,\zeta)
   &=A_1\overline{A_2}|\eta|^{\frac{k-d+1}{2}} |\zeta|^{\frac{1-k}2} 
    \mathcal R_{\alpha_p,\beta_p}(h)(-|\eta|,|\zeta|). \notag
 \end{align} 
 Here, $\alpha_p,\beta_p$ were defined in \eqref{eq:DefAlBe}, and
\begin{equation}\label{eq:hdefh}
h(\rho_1,\rho_2):=\rho_1^{\frac{d-k-1}{2}}\rho_2^{\frac{k-1}{2}}(1+\rho_1)^{\alpha_p} (1+\rho_2)^{\beta_p} f_0(\rho_1,\rho_2)\mathbbm1_{[1,\infty)^2}(\rho_1,\rho_2).
\end{equation}
  
\begin{proposition}\label{prop:f4}
For every $1\leq p\leq \frac{2(d+m)}{d+m+2}$, there exists $C=C(k,d,p)<\infty$ such that 
\[
  \|\widehat{f_4}\|_{L^2(\sph{d-1})} + 
  \|\widehat{f_5}\|_{L^2(\sph{d-1})} \leq C\|f\|_{L^p(\R^d)},
\]
for every $G_k$-symmetric Schwartz function $f:\R^d\to\Co$.
\end{proposition}
\begin{proof} 
By the usual considerations, it suffices to bound $\|\widehat{f_4}\|_{L^2(\sph{d-1})}$.
Identity \eqref{eq:f4IntermsR} and Lemma~\ref{lem:SphInt} together imply 
 \begin{align*}
  \int_{\sph{d-1}} |\widehat f_4(\eta,\zeta)|^2\d\sigma(\eta,\zeta)
   \lesssim \int_{\sph{1}} 
   \left|\mathcal R_{\alpha_p,\beta_p}(h)(\omega)\right|^2 \d\sigma(\omega),
\end{align*}
where the function $h$ was defined in \eqref{eq:hdefh}.
 For every $1\leq p\leq \frac{2(d+m)}{d+m+2}$, Corollary~\ref{cor:WeightedCor2} and \eqref{eq:Norm} together yield
\begin{align*}
  \int_{\sph{d-1}} |\widehat f_4(\eta,\zeta)|^2\d\sigma(\omega)
  \les  \|h\|_{L^p(\R^2)}^2
  \les  \| \rho_1^{\frac{d-k-1}{p}}\rho_2^{\frac{k-1}{p}}f_0(\rho_1,\rho_2)\|_{L^p_{\rho_1,\rho_2}(\R^2_+)}^2  
  \simeq  \|f\|_{L^p(\R^d)}^2.
\end{align*}
In the second estimate, we used the fact that the support of $h$ is contained in
$[1,\infty)^2\subset\R^2$, so that $1+\rho_j \sim \rho_j$, for each $j\in\{1,2\}$. 
This concludes the proof of Proposition~\ref{prop:f4}.
\end{proof}

\subsection{Conclusion of the proof}
The aim is to verify estimate \eqref{eq:GoalThm1} for every $G_k$-symmetric Schwartz function $f:\R^d\to\Co$.
Splitting $f$ as in \eqref{eq:splitting}, by the subsequent considerations it suffices to bound
{$\|\widehat{f}_j\|_{L^2(\sph{d-1})},1\leq j\leq 5$, appropriately in terms of $\|f\|_{L^p(\R^d)}$,
whenever $1\leq p\leq \frac{2(d+m)}{d+m+2}$}. 
In turn, this is accomplished by Propositions \ref{prop:f1}, \ref{prop:f2}, \ref{prop:f4}.
The proof of Theorem~\ref{thm:main} is thus complete.


\section{Proof of Theorem \ref{thm:G_kExistence}}\label{sec:ProofThmExistence}

In this section, we explain how Theorem \ref{thm:G_kExistence} follows from Theorem \ref{thm:main}.
As in most optimization problems, the difficulty is to find a weak limit of a maximizing sequence which is non-zero.
The key observation is that, in light of estimate \eqref{eq:NewTS}, the Stein--Tomas exponent
$p_d=\frac{2(d+1)}{d+3}$ {\it is no longer an endpoint exponent within the class of $G_k$-symmetric
functions}, provided $2\leq k\leq d-2$.
Therefore Theorem \ref{thm:main} can be used in conjunction with the Fourier decay property from Corollary
\ref{cor:decay} to show that maximizing sequences do not converge weakly to zero.
\medskip

Precompactness of maximizing sequences for non-endpoint, $L^2$-based adjoint restriction estimates is in general well-understood.
We follow the approach of  \cite{COSS19} which relies on a useful reformulation of the Br\'ezis--Lieb lemma \cite{BL83} given in \cite{FVV11}, but with an important twist.
Since translations are not symmetries of the $G_k$-symmetric problem, one may expect precompactness to hold, instead of precompactness modulo translations as in the general non-symmetric case \cite{CS12, FLS16}.
To facilitate the comparison with  references \cite{COSS19, FVV11}, and especially \cite{FLS16}, we choose to formulate and prove Theorem \ref{thm:G_kExistence} for the extension operator instead. 
Thus we are led to define the Hilbert space\footnote{That $L^2_{G_k}(\sph{d-1})$ is indeed a Hilbert space follows from the Riesz--Fischer Theorem, since $G_k$-symmetry is preserved under pointwise limits.
} $L^2_{G_k}(\sph{d-1}):=\{F\in L^2(\sph{d-1}): F \text{ is } G_k\text{-symmetric}\}$, and the quantity
\[
  {\bf T}^\ast_{d,k}(p) := \sup_{0\neq F\in L_{G_k}^{2}(\sph{d-1})} \frac{\|\widehat
   {F\sigma}\|_{L^{p'}(\R^d)}}{\|F\|_{L^2(\sph{d-1})}}.
 \]
 By duality, we naturally have that ${\bf T}_{d,k}^\ast(p)={\bf T}_{d,k}(p)$, but we will use both
 designations in order to keep track of the extremal problem under consideration.

\begin{theorem}\label{thm:G_kExistenceAdj}
Let $d\geq 4$, $k\in\{2,3,\ldots,d-2\}$, $m:=\min\{d-k,k\}$, and $1\leq p< \frac{2(d+m)}{d+m+2}$. Then
maximizing sequences for ${\bf T}_{d,k}^\ast(p)$, normalized in 
$L^2(\sph{d-1})$, are precompact in $L_{G_k}^{2}(\sph{d-1})$.
In particular, maximizers for ${\bf T}_{d,k}^\ast(p)$ exist.
\end{theorem}

\noindent Theorem \ref{thm:G_kExistenceAdj} is in fact equivalent to Theorem \ref{thm:G_kExistence} via a well-known
argument; see \cite[\S~6]{St20} for a more general statement along these lines.
For the convenience of the reader, we present the details of the relevant implication in Appendix \ref{ref:appendix}. 
The proof of Theorem \ref{thm:G_kExistenceAdj} relies on \cite[Proposition 1.1]{FVV11}, which is
a useful reformulation of the Br\'ezis--Lieb lemma \cite[Theorem 1]{BL83}.

\begin{proof}[Proof of Theorem \ref{thm:G_kExistenceAdj}]
Let $(F_n)_{n\in\N}\subset L_{G_k}^2(\sph{d-1})$ be an $L^2$-normalized maximizing sequence for ${\bf
T}_{d,k}^\ast(p)$, i.e., $F_n$ is $G_k$-symmetric, $\|F_n\|_{L^2(\sph{d-1})}=1$ for all $n\in\N$, and
\begin{equation}\label{eq:E1}
\|\widehat{F_n\sigma}\|_{L^{p'}(\R^d)}\to{\bf T}_{d,k}^\ast(p), \text{ as }n\to \infty.
\end{equation}
From Theorem \ref{thm:main} in dual form and $L^2$-normalization of $(F_n)_{n\in\N}$, there exists $C_{k,d}<\infty$ such that
\begin{equation}\label{eq:E2}
\sup_{n\in\N} \| \widehat{F_n\sigma} \|_{L^{p_\star'}(\R^d)} <C_{k,d},
\end{equation}
where $p_\star=\frac{2(d+m)}{d+m+2}$ and $m=\min\{d-k,k\}$.
By convexity of $L^p$-norms, we further have
\begin{equation}\label{eq:E3}
\|\widehat{F_n\sigma}\|_{L^{p'}(\R^d)}\leq \|\widehat{F_n\sigma}\|_{L^{p_\star'}(\R^d)}^\theta
\|\widehat{F_n\sigma}\|_{L^\infty(\R^d)}^{1-\theta},
\end{equation}
where 
$\theta\in (0,1)$ is given by $\frac{1}{p'}=\frac{\theta}{p_\star'}+\frac{1-\theta}{\infty}$, i.e.,
$\theta=\frac{p_\star'}{p'}$.
Estimates \eqref{eq:E1}, \eqref{eq:E2}, \eqref{eq:E3} together imply the existence of $\varepsilon_0=\varepsilon_0(k,d)>0$, depending only on $k,d$, for which
$\|\widehat{F_n\sigma}\|_{L^\infty(\R^d)} \geq \varepsilon_0$
for every $n\in \N$ (possibly after extraction of a subsequence).
Thus there exists a sequence $(x_n)_{n\in\N}\subset\R^d$, such that 
\begin{equation}\label{eq:nonzero}
|\widehat{F_n\sigma}(x_n)|\geq \varepsilon_0>0, \text{ for every } n\in\N.
\end{equation}
Corollary \ref{cor:decay} guarantees\footnote{Note the tension between estimates \eqref{eq:Pointwise} and \eqref{eq:nonzero} which, in particular, reveals that a maximizing sequence for ${\bf T}_{d,k}^\ast{(p)}$ cannot {\it concentrate} on a copy of $\sph{\min\{d-k,k\}-1}$ inside $\sph{d-1}$. If $k\in\{1,d-1\}$, this would amount to concentration at a pair of antipodal points on $\sph{d-1}$, which  in \cite{CS12, FLS16} was identified as the ``most essential obstacle'' to the precompactness of maximizing sequences modulo symmetries {when $p=p_d$}.} the existence of a certain radius $R=R(k,d)<\infty$, depending only on $k,d$, for which
 $|x_n|\leq R$ for every $n\in\N$.
On the other hand, by the Cauchy--Schwarz inequality together with the $L^2$-normalization of $(F_n)_{n\in\N}$, we have that 
\begin{align*}
&\|\widehat{F_n\sigma}\|_{L^\infty(\R^d)}
 \leq \|F_n\|_{L^1(\sph{d-1})}
 {\leq \sigma(\sph{d-1})^{\frac{1}{2}}\|F_n\|_{L^2(\sph{d-1})}= 
 \sigma(\sph{d-1})^{\frac{1}{2}}};\\
&\|\nabla_x(\widehat{F_n\sigma})\|_{L^\infty(\R^d)}\leq\|F_n|\cdot|\|_{L^1(\sph{d-1})}\leq \|F_n\|_{L^1(\sph{d-1})}
{\leq  \sigma(\sph{d-1})^{\frac{1}{2}}}.
 \end{align*}
It follows that the sequence $(\widehat{F_n\sigma})_{n\in\N}$ is uniformly bounded and equicontinuous on the cube $Q_R:=[-R,R]^d\subset\R^d$. The Arzel\`a--Ascoli Theorem on compact subsets of $\R^d$ then implies that the sequence $(\widehat{F_n\sigma})_{n\in\N}$ has a subsequence which converges uniformly to a limit in $Q_R$. That this limit is nonzero follows at once from \eqref{eq:nonzero} and the fact that $(x_n)_{n\in\N}\subset Q_R$.

\medskip

Now, since the sequence $(F_n)_{n\in\N}$ is bounded on $L_{G_k}^2(\sph{d-1})$, it has a weakly convergent subsequence. 
In other words, there exists a function $F_\star\in L^2_{G_k}(\sph{d-1})$, such that $F_n\rightharpoonup F_\star$ weakly in $L^2_{G_k}(\sph{d-1})$, as $n\to\infty$.
Since the extension operator is bounded from $L^2_{G_k}(\sph{d-1})$ to $L^{p'}(\R^d)$, it follows that
$\widehat{F_n\sigma}\rightharpoonup \widehat{F_\star\sigma}$ weakly in $L^{p'}(\R^d)$, as $n\to\infty$.
Since uniform convergence implies weak convergence, and weak limits are unique, 
from the previous paragraph it follows that $\widehat{F_\star\sigma}$ is nonzero, and so the function $F_\star$ is itself nonzero.

\medskip

We are now in a position to apply \cite[Proposition 1.1]{FVV11} to the extension operator on $\sph{d-1}$ with
$\mathcal H=L^2_{G_k}(\sph{d-1})$, {{$p=p_d'\in(2,\infty)$}}, $(h_n)_{n\in\N}=(F_n)_{n\in\N}$, and
$\overline h=F_\star$.
Hypotheses (1) and (2) from \cite[Proposition 1.1]{FVV11} hold by the assumptions on the sequence
$(F_n)_{n\in\N}$, and hypothesis (3) follows from the previous paragraph.
Finally, hypothesis (4) is an easy consequence of the compactness of $\sph{d-1}$.
The conclusion is that, possibly after extraction of a subsequence, $F_n\to F_\star$ in $L^2_{G_k}(\sph{d-1})$, as $n\to\infty$.
In particular, $F_\star$ is a maximizer for ${\bf T}_{d,k}^\ast(p)$.
This finishes the proof of the theorem.
\end{proof} 

\section{Proof of Theorem \ref{thm:Nec}}\label{sec:ProofThm2}

In this section, we construct appropriate examples to show the necessity of conditions (i)--(iii) in the statement of Theorem \ref{thm:Nec}.
We work with the extension operator \eqref{eq:ExtensionOp} rather than with the restriction operator directly, and aim to show that the estimate
\begin{equation*}
\|\widehat{F\sigma}\|_{L^{p'}(\R^d)}\leq C(k,d,p,q)  \|F\|_{L^{q'}(\sph{d-1})},
\end{equation*}
which is dual to \eqref{eq:GkRestriction},
can only hold for every $G_k$-symmetric function $F:\sph{d-1}\to\Co$ provided 
$d,p,q$, and $m=\min\{d-k,k\}$ are chosen in such a way that conditions (i)--(iii) in the statement of Theorem \ref{thm:Nec} hold.
As in the general non-symmetric situation,  the first condition $\frac{d+1}{2d}<\frac1p$ is dictated by the choice $F\equiv 1$, since 
\begin{equation}\label{eq:IdNec}
\widehat{\sigma}\in L^{p'}(\R^d) \text{ if and only if } p'>\frac{2d}{d-1}.
\end{equation}
The latter equivalence follows from identity~\eqref{eq:SigmaHat} together with the standard asymptotics of
Bessel functions at zero and infinity; recall \eqref{eq:BesselAsymp}--\eqref{eq:BesselAsympNear0}.

\medskip

The remaining necessary conditions are obtained from analyzing a $G_k$-symmetric variant of Knapp's
construction. Let $d\geq 4$, $k\in\{2,3,\ldots,d-2\}$ be given, and assume without loss of generality that $m=k$, which we
take as fixed from now onwards.
Given $\delta\in(0,\frac12)$, consider the following union of two ``spherical caps'' of radius $\delta$:  
\[
  \mathcal C_{\delta}:=\{(\eta,\zeta)\in\R^{d-k}\times\R^k: |\eta|^2+|\zeta|^2=1,
  |\eta|<\delta\}\subset\sph{d-1}.
\] 
By construction, the set $\mathcal C_{\delta}$ is $G_k$-symmetric,
and  so are the indicator function $\mathbbm{1}_{\delta}:=\mathbbm{1}_{\mathcal C_{\delta}}$ and its
Fourier extension, $\widehat{\mathbbm{1}_{\delta}\sigma}$.
Using Lemma~\ref{lem:SphInt}, we estimate: 
\begin{equation}\label{eq:DenEst}
\|\mathbbm{1}_{\delta}\|_{L^{q'}(\sph{d-1})}
 = \sigma(\mathcal C_{\delta})^{\frac1{q'}}
\simeq \left(\int_0^\delta r^{d-k-1} (1-r^2)^{\frac{k-2}2} \d r\right)^{\frac1{q'}}
\simeq \delta^{\frac{d-k}{q'}}.
\end{equation}
On the other hand, if $(y,z)\in\R^{d-k}\times\R^k$, then \eqref{eq:SigmaHat} together with a further
application of Lemma~\ref{lem:SphInt} yield:
\begin{multline}\label{eq:PreLB}
\widehat{\mathbbm{1}_{\delta}\sigma}(y,z)
=\int_{\mathcal C_{\delta}} e^{i(y,z)\cdot(\eta,\zeta)}\d\sigma(\eta,\zeta)\\
\simeq \int_0^\delta r^{d-k-1} (1-r^2)^{\frac{k-2}2} (r|y|)^{\frac{2-d+k}2} J_{\frac{d-k-2}2}(r|y|) (\sqrt{1-r^2}|z|)^{\frac{2-k}2} J_{\frac{k-2}2}(\sqrt{1-r^2}|z|)\d r.
\end{multline}
Let $\{z_j\}_{j\geq 1}$ denote the increasing sequence of local maxima of the Bessel function
$J_{(k-2)/2}$.
By the asymptotic expansion \eqref{eq:BesselAsymp}, there exist constants $0<c,C<\infty$, such that
$|z_j-2\pi j|\leq C$, as $j\to\infty$, and moreover
$J_{\frac{k-2}2}(z_j)\geq c j^{-\frac12}$,
 for every $j\geq 0$.
Recalling \eqref{eq:BesselAsympNear0} and shrinking $c$ if necessary,  we obtain  
\begin{align}
&z^{\frac{2-d+k}2} J_{\frac{d-k-2}2}(z)\geq c,\qquad \text{ for every } z\in(0,c),\label{eq:Est2}\\
&z^{\frac{2-k}2} J_{\frac{k-2}2}(z)\geq cj^{\frac{1-k}2},\qquad 
  \text{ for every } z\in[z_j-c,z_j+c] \text{ and } j\geq 0.\label{eq:Est1}
  \end{align}
Consider the disjoint union $E:=\bigcup_{j=1}^{\lfloor c\delta^{-2}\rfloor} E_{j}$, where
each set $E_{j}$ is defined as follows:
\[E_{j}:=\left\{(y,z)\in\R^{d-k}\times\R^k: 0\leq |y|\leq c\delta^{-1},\frac{z_j-c}{\sqrt{1-\delta^2}}\leq |z|\leq z_j+c\right\}.\]
Each set $E_{j}$ is $G_k$-symmetric, and so is $E$.
If $(y,z)\in E$, then estimates {{\eqref{eq:Est2}--\eqref{eq:Est1}}} applied to \eqref{eq:PreLB} imply the following lower bound:
\[
  |\widehat{\mathbbm{1}_{\delta}\sigma}(y,z)|\gtrsim \delta^{d-k} j^{\frac{1-k}2},
\]
provided $\delta\in(0,c)$ is chosen sufficiently small. 
On the other hand, there exists an index $j_0$
 with the following property:
for each $j\in\{j_0,j_0+1,\ldots,\lfloor c\delta^{-2}\rfloor\}$, 
\[(z_j+c)^k-\left(\frac{z_j-c}{\sqrt{1-\delta^2}}\right)^k\gtrsim  j^{k-1},\]
provided $\delta>0$ is chosen sufficiently small. 
This follows directly from Taylor expansion, and readily implies the size estimate 
$|E_j|\gtrsim \delta^{k-d}j^{k-1}$,
 for each $j_0\leq j\leq \lfloor c\delta^{-2}\rfloor$. 
As a consequence,
\begin{align*}
\|\widehat{\mathbbm{1}_{\delta}\sigma}\|_{L^{p'}(\R^d)}^{p'}
&\geq\|\widehat{\mathbbm{1}_{\delta}\sigma}\|_{L^{p'}(E)}^{p'}
\gtrsim \sum_{j=j_0}^{\lfloor c\delta^{-2}\rfloor} (\delta^{d-k}j^{\frac{1-k}2})^{p'} |E_{j}|\\ 
&\gtrsim\sum_{j=j_0}^{\lfloor c\delta^{-2}\rfloor} (\delta^{d-k}j^{\frac{1-k}2})^{p'} (\delta^{k-d}j^{k-1}) 
 \\
&=\delta^{(d-k)(p'-1)}\sum_{j=j_0}^{\lfloor c\delta^{-2}\rfloor} j^{-\frac{(k-1)(p'-2)}2}.
\end{align*}
According to $\frac{(k-1)(p'-2)}2$ being smaller than, equal to, or larger than $1$, we thus
obtain
\begin{displaymath}
\|\widehat{\mathbbm{1}_{\delta}\sigma}\|_{L^{p'}(\R^d)} \gtrsim \left\{ \begin{array}{ll}
\delta^{\frac{d+k}p-k-1}, & \textrm{if $\frac1p<\frac{k+1}{2k}$},\\
\delta^{\frac{d+k}p-k-1}|\log(\delta)|^{\frac1{p'}}, & \textrm{if $\frac1p=\frac{k+1}{2k}$},\\
\delta^{\frac{d-k}p}, & \textrm{if $\frac1p>\frac{k+1}{2k}$}.
\end{array} \right.
\end{displaymath}
The latter estimate is valid also when $p'=\infty$. 
Together with \eqref{eq:DenEst}, we finally conclude:
\begin{displaymath}
\frac{\|\widehat{\mathbbm{1}_{\delta}\sigma}\|_{L^{p'}(\R^d)}}{\|\mathbbm{1}_{\delta}\|_{L^{q'}(\sph{d-1})}} \gtrsim \left\{ \begin{array}{ll}
\delta^{\frac{d+k}p+\frac{d-k}q-d-1}, & \textrm{if $\frac1p<\frac{k+1}{2k}$},\\
\delta^{\frac{d+k}p+\frac{d-k}{q}-d-1}|\log(\delta)|^{\frac1{p'}}, & \textrm{if $\frac1p=\frac{k+1}{2k}$},\\
\delta^{\frac{d-k}p+\frac{d-k}{q}-d+k}, & \textrm{if $\frac1p>\frac{k+1}{2k}$}.
\end{array} \right.
\end{displaymath}
The proof of Theorem \ref{thm:Nec} is completed by letting $\delta\to 0^+$.

\begin{remark}\label{rem:G1Knapp}
  \emph{If  $m=\min\{k,d-k\}=1$, then the following $G_1$-symmetric version of Knapp's construction  
  reveals that estimates beyond those predicted by the restriction conjecture, recall \eqref{eq:FRC}--\eqref{eq:FRCNumerology}. are not possible within the 
  $G_1$-symmetric setting. 
  Define    
  \[\mathcal C_\delta:=\{(\eta,\zeta)\in\R^{d-1}\times\R: |\eta|^2+\zeta^2=1, |\eta|<\delta
  \}\subset\sph{d-1}\] and 
  $E:=\bigcup_{j=1}^{\lfloor(2\delta)^{-2}\rfloor} E_j$, 
  where  
  \[E_j:=\left\{(y,z)\in\R^{d-1}\times\R: 0\leq |y|\leq\frac{\pi}{4\delta},\frac{2\pi
  j-\frac{\pi}4}{\sqrt{1-\delta^2}}\leq |z|\leq2\pi j+\frac{\pi}4\right\}.\] 
  Here, $\delta>0$ is a sufficiently small parameter, and the values $\{2\pi j\}_{j\geq 1}$ 
  are the counterparts of the Bessel maxima $\{z_j\}_{j\geq 1}$ considered above. 
  Naturally, the sets $\mathcal C_\delta$ and $E$ are both $G_1$-symmetric.
  Repeating the steps from the proof of Theorem \ref{thm:Nec}, one
  finds that $|E_j|\gtrsim \delta^{1-d}$ and thus
  $\|\widehat{\mathbbm{1}_\delta\sigma}\|_{L^{p'}(\R^d)}^{p'}\gtrsim\delta^{(d-1)p'-(d+1)}$.
  In turn, this implies the following lower bound:
  \[\frac{\|\widehat{\mathbbm{1}_\delta\sigma}\|_{L^{p'}(\R^d)}}{\|\mathbbm{1}_\delta\|_{L^{q'}(\sph{d-1})}}
  \gtrsim \delta^{\frac{d+1}{p}+\frac{d-1}{q}-d-1}.\]
  The latter quotient remains bounded as  $\delta\to 0^+$ if and only if
  $\frac{d+1}{p}+\frac{d-1}{q}\geq d+1$, which matches the second condition in \eqref{eq:FRCNumerology}. 
  All in all, we recover the same necessary conditions as in the general, non-symmetric case.   }
\end{remark}


\section{Proof of Theorem \ref{thm:easy}}\label{sec:ProofThm3}

In this section, we provide a short proof of Theorem \ref{thm:easy}.
No generality is lost in assuming $m=k$, $f\in\mathcal S(\R^d)$, and
throughout the proof we set $p:=\frac{2k}{k+1}$. The representation formula from
Lemma~\ref{lem:GkSymFT} and the bounds \eqref{eq:BesselAsymp}--\eqref{eq:BesselAsympNear0} for Bessel functions together imply
\begin{align*}
|\widehat{f}(\eta,\zeta)|
&\simeq |\eta|^{\frac{2-d+k}{2}} |\zeta|^{\frac{2-k}2} 
\left|\int_0^\infty \int_0^\infty \rho_1^{\frac{d-k}2}\rho_2^{\frac k2}  f_0(\rho_1,\rho_2) J_{\frac{d-k-2}{2}}(\rho_1|\eta|)J_{\frac{k-2}2}(\rho_2|\zeta|)\d\rho_1\d\rho_2\right|\\
&\lesssim  \int_0^\infty\int_0^\infty \rho_1^{d-k-1}\rho_2^{k-1}|f_0(\rho_1,\rho_2)|(1+\rho_1|\eta|)^{\frac{k-d+1}2}(1+\rho_2|\zeta|)^{\frac{1-k}2} \d\rho_1\d\rho_2.
\end{align*}
{{If $k<\frac d2$, then}} H\"older's inequality in Lorentz spaces \cite[Theorem 3.4]{ON63} implies the following pointwise bound for $|\widehat{f}(\eta,\zeta)|$:
\begin{align}
\|\rho_1^{\frac{d-k-1}{p}}\rho_2^{\frac{k-1}p}&f_0(\rho_1,\rho_2)\|_{L^{p,1}_{\rho_1,\rho_2}(\R_+^2)}
\|\rho_1^{\frac{d-k-1}{p'}}\rho_2^{\frac{k-1}{p'}}(1+\rho_1|\eta|)^{\frac{k-d+1}2}(1+\rho_2|\zeta|)^{\frac{1-k}2}\|_{L^{p',\infty}_{\rho_1,\rho_2}(\R_+^2)}\notag\\
&= \|f\|_{L^{p,1}(\R^d)} \|(1+|y||\eta|)^{\frac{k-d+1}2}(1+|z||\zeta|)^{\frac{1-k}2}\|_{L^{p',\infty}(\R^d)}\notag\\
&\simeq \|f\|_{L^{p,1}(\R^d)} |\eta|^{\frac{k-d}{p'}}|\zeta|^{-\frac{k}{p'}},\label{eq_Est1}
\end{align}
where the last line follows from changing variables $y\rightsquigarrow |\eta|y\in\R^{d-k}$ and $z\rightsquigarrow
|\zeta|z\in\R^k$, and using the facts that  $2k{{<}} d$ and $p'=\frac{2k}{k-1}$ in order to control the corresponding
weak quasi-norm. 
{{If $k=\frac d2$, then Hölder's inequality in mixed Lorentz spaces \cite[Prop.\@ 6.1]{LF77} implies\footnote{Recall the definition \eqref{eq_Xp} of the space $X_p$.} the following pointwise bound for $|\widehat f(\eta,\zeta)|$:
\begin{align}
\|\rho_1^{\frac{d-k-1}{p}}\rho_2^{\frac{k-1}p}&f_0(\rho_1,\rho_2)\|_{L_{\rho_1}^{p,1}(\R_+; L_{\rho_2}^{p,1}(\R_+))}\notag\\
&\times\|\rho_1^{\frac{d-k-1}{p'}}\rho_2^{\frac{k-1}{p'}}(1+\rho_1|\eta|)^{\frac{k-d+1}2}(1+\rho_2|\zeta|)^{\frac{1-k}2}\|_{L_{\rho_1}^{p',\infty}(\R_+; L_{\rho_2}^{p',\infty}(\R_+))}\notag\\
&\lesssim \|f\|_{X_p} \|(1+|y||\eta|)^{\frac{k-d+1}2}(1+|z||\zeta|)^{\frac{1-k}2}\|_{L^{p',\infty}(\R^{d-k}; L^{p',\infty}(\R^k))}\notag\\
&\simeq \|f\|_{X_p} |\eta|^{\frac{k-d}{p'}}|\zeta|^{-\frac{k}{p'}}.\label{eq_Est2}
\end{align}
Estimates \eqref{eq_Est1}--\eqref{eq_Est2}}} can be integrated over the unit sphere, resulting in
\begin{align*}
\|\widehat{f}\|_{L^{p',\infty}(\sph{d-1})}
\lesssim \|f\|_{{{ X_p }}} \||\eta|^{\frac{k-d}{p'}}|\zeta|^{-\frac{k}{p'}}\|_{L^{p',\infty}(\sph{d-1})}
\lesssim \|f\|_{{{ X_p }}}.
\end{align*}
This concludes the proof of Theorem~\ref{thm:easy}. 


\section*{Acknowledgements}
RM\@ is funded by the Deutsche Forschungsgemeinschaft (DFG, German Research Foundation) -- Project-ID 258734477 -- SFB 1173.
DOS\@ is supported by the EPSRC New Investigator Award ``Sharp Fourier Restriction Theory'', grant no.\@ EP/T001364/1,
and the DFG under Germany's Excellence Strategy -- EXC-2047/1 -- 390685813,
and is grateful to Giuseppe Negro, {Andreas Seeger} and Mateus Sousa for valuable discussions during the preparation of this work.
{The authors thank the anonymous referee for carefully reading the manuscript and valuable suggestions.}


\appendix
\section{Theorem \ref{thm:G_kExistenceAdj} implies Theorem \ref{thm:G_kExistence}}\label{ref:appendix}
Let $(f_n)_{n\in\N}\subset L^{p}_{G_k}(\R^d)$ be a maximizing sequence for ${\bf T}_{d,k}(p)$,
normalized in $L^{p}(\R^d)$.
In other words, each $f_n$ is $G_k$-symmetric in $\R^d$, $\|f_n\|_{L^{p}(\R^d)}=1$ for every $n\in\N$,
and
\begin{equation}\label{eq:fnMaxSeq}
\|\widehat{f}_n\|_{L^2(\sph{d-1})}\to{\bf T}_{d,k}(p), \text{ as } n\to\infty.
\end{equation}
Let $F_n:=\|\widehat f_n\|_{L^2(\sph{d-1})}^{-1}\widehat{f}_n\mid_{\sph{d-1}}$. 
Then $(F_n)_{n\in\N}\subset L^2_{G_k}(\sph{d-1})$ is an $L^2$-normalized maximizing sequence for ${\bf
T}_{d,k}^\ast(p)$, i.e.\@ each $F_n$ is $G_k$-symmetric on $\sph{d-1}$, $\|F_n\|_{L^2(\sph{d-1})}=1$ for
every $n\in\N$, and $\|\widehat{F_n\sigma}\|_{L^{p'}(\R^d)}\to {\bf T}_{d,k}^\ast(p), \text{ as
}n\to\infty$.
By Theorem \ref{thm:G_kExistenceAdj}, there exists $F_\star\in L^2_{G_k}$ such that, possibly after
extraction of a subsequence, $F_n\to F_\star$ in $L^2(\sph{d-1})$, as $n\to\infty$. In particular, observe
that $\|F_\star\|_{L^2(\sph{d-1})}=1$ and
\begin{equation}\label{eq:InTheLimit}
\|\widehat{F_\star\sigma}\|_{L^{p'}(\R^d)}={\bf T}_{d,k}^\ast(p).
\end{equation} 
Since the sequence $(f_n)_{n\in\N}$ is bounded on $L^{p}_{G_k}(\R^d)$, it has a weakly convergent
subsequence.
In other words, there exists a function $f_\star\in L_{G_k}^{p}(\R^d)$, such that $f_n\rightharpoonup
f_\star$ weakly in $L^{p}_{G_k}(\R^d)$, as $n\to\infty$.
We claim that $f_\star$ is a maximizer for ${\bf T}_{d,k}(p)$, and that in fact $f_n\to f_\star$ strongly
in $L^{p}_{G_k}(\R^d)$, as $n\to\infty$.
To see this, note that
\begin{multline}\label{eq:chain}
{\bf T}_{d,k}(p)^2
=\lim_{n\to\infty} \|\widehat{f}_n\|_{L^2(\sph{d-1})}^2
={\bf T}_{d,k}(p)\lim_{n\to\infty} |\langle f_n, \widehat{F_n\sigma}\rangle|\\
= {\bf T}_{d,k}(p)|\langle f_\star,\widehat{F_\star\sigma}\rangle|
\leq {\bf T}_{d,k}(p)\|f_\star\|_{L^{p}(\R^d)} \|\widehat{F_\star\sigma}\|_{L^{p'}(\R^d)}
=\|f_\star\|_{L^{p}(\R^d)} {\bf T}^2_{d,k}(p).
\end{multline}
Here we used \eqref{eq:fnMaxSeq}, duality, weak convergence of $(f_n)_{n\in\N}$ and continuity of the
extension operator, H\"older's inequality, and \eqref{eq:InTheLimit} together with ${\bf
T}_{d,k}^\ast(p)={\bf T}_{d,k}(p)$. From the chain of inequalities \eqref{eq:chain}, we read off that
\begin{equation}\label{eq:NormLimit}
\|f_\star\|_{L^{p}(\R^d)}\geq 1=\lim_{n\to\infty} \|f_n\|_{L^p(\R^d)}.
\end{equation}
Since the reverse inequality holds since $f_n\rightharpoonup f_\star$ weakly in $L^{p'}_{G_k}(\R^d)$, as
$n\to\infty$, we actually have equality in \eqref{eq:NormLimit}.
But weak convergence together with convergence of norms implies strong convergence; see \cite[Theorem 2.11]{LL01}. 
Therefore $f_n\to f_\star$ in $L^p_{G_k}(\R^d)$, as $n\to\infty$.
By continuity of the restriction operator, it follows that $f_\star$ is a maximizer for ${\bf
T}_{d,k}(p)$, as desired.
This concludes the proof that Theorem \ref{thm:G_kExistenceAdj} implies Theorem \ref{thm:G_kExistence}.


\end{document}